\documentclass[a4paper,12pt]{article}
\usepackage{a4wide}
\usepackage[utf8]{inputenc}
\usepackage{amsmath,amssymb,amsthm, amsfonts,color}
\usepackage{amsmath,amsthm,amssymb,amsfonts}
\usepackage[colorlinks=true,citecolor=black,linkcolor=black]{hyperref}
\usepackage{ upgreek }
\usepackage{authblk}

\usepackage[dvips]{graphicx}
\usepackage{constants}

\newtheorem{mainthm}{Theorem} 

\newtheorem*{verjov}{The Verjovsky Conjecture}
\newtheorem*{thmm}{Theorem}   
\newtheorem{thm}{Theorem}
\newtheorem{cor}{Corollary}
\newtheorem{claim}{Claim}

\newtheorem{lem}[thm]{Lemma}
\newtheorem{definition}[thm]{Definition}
\newtheorem{prop}[thm]{Proposition}
\newtheorem{rem}[thm]{Remark}

\newcommand{\jac}{\operatorname{Jac}} 
\newcommand{\flow}[2]{{\phi^{#2}_{#1}}} 
\newcommand{\dm}[1]{D#1} 
\newcommand{\pb}[1]{\left(\smash{#1}\right)^{*}} 
\newcommand{\cM}{{M}} 
\newcommand{\cC}{\mathcal{C}} 
\newcommand{\bE}{\mathbb{E}} 
\newcommand{\cL}{\mathcal{L}} 
\newcommand{\bN}{\mathbb{N}} 
\newcommand{\norm}[1]{\left\|#1\right\|} 
\newcommand{\bR}{\mathbb{R}} 
\newcommand{\extd}{d} 

\begin{document}

\title{Proof of the Verjovsky Conjecture}
\author{Khadim War}
\date{}

\affil{Instituto Nacional de Matematica Pura e Aplicada (IMPA)}
\affil{khadim@impa.br}

\maketitle
\begin{abstract}
In this paper we present a proof of the Verjovsky conjecture: Every codimension-one Anosov flow on a manifold of dimension greater than three is topologically equivalent to the suspension of a hyperbolic toral automorphism. 
In fact, the conjecture is derived from possible more general result that says 
 that for every \(\cC^{4}\) codimension-one volume-preserving Anosov flow on a manifold of dimension greater than three,  a suitable time change guarantees that the stable and unstable sub-bundles are then jointly integrable.
\end{abstract}

\footnote{With pleasure, we thank  Andrei Agrachev, Thierry Barbot and Stefano Luzzatto for listening the first time to the proof. We thank Keith Burns, Boris Hasselblatt, Rafael Potrie, Mark Pollicott, Raul Ures, Marcelo Viana and  Amie Wilkinson for encounraging comments. We also thank Federico Rodriguez-Hertz for bringing my attention for the first time on this Conjecture. We thank Oliver Butterley with who we discussed many techniques used here.}

\tableofcontents

\section{Introduction and results}

Anosov flows are central examples of chaotic dynamical systems.
Let \(M\) be a compact connected Riemannian manifold. A flow \(F:=\{f^{t}, t\in\mathbb{R}\}: M\to M\) is called
\emph{Anosov} if it is \emph{uniformly hyperbolic} in the sense the tangent bundle splits into three invariant sub-bundles \(\bE^{c}\), \(\bE^s\) and \(\bE^u\), where $\bE^{c}$ is one-dimensional and contains the flow direction and where the vectors in \(\bE^s\) (resp.\@  \(\bE^u\)) are exponentially contracted (resp.\@ expanded).
The sub-bundles \(\bE^s\) and \(\bE^u\) are referred to as stable and unstable bundles.
The three main examples of Anosov flows are: (1) Suspensions over Anosov diffeomorphisms; (2) Geodesic flows on manifolds of negative curvature; (3) Small perturbations of such geodesic flows (Anosov are structurally stable but typically the perturbed flow will fail to be a geodesic flow).

Two flows are said to be \emph{topologically equivalent} if there exists a homeomorphism which maps orbits to orbits homeomorphically and preserves orientation of the orbits.
It is natural to ask for a classification of Anosov flows up to topological equivalence. 
An Anosov flow is said to be \emph{codimension-one} if the stable or unstable bundle is one-dimensional. A flow admits a \text{global cross section} if there exists a closed co-dimension one manifold which intersects every orbit transversally and consequently the manifold is topologically a suspension maniflod. The geodesic flow on a negatively curved compact surface is an example of a three-dimensional Anosov flow with no global cross-section (if it had a global cross-section it would contradict a result of Plante~\cite[Theorem 4.8]{Pla72}). Consequently it cannot be topologically equivalent to the suspension of an Anosov diffeomorphism. However geodesic flows are never codimension-one in higher dimensions.
In the 1970s Verjovsky made a conjecture concerning this question in higher dimensions.

\begin{verjov} 
Any codimension-one Anosov flow on a closed manifold of dimension greater than three is topologically equivalent to the suspension flow of a hyperbolic toral automorphism.
\end{verjov}
If true this conjecture would give a complete classification of codimension-one Anosov flows, up to topological equivalence, in higher dimensions.
Plante~\cite{Pla81} proved the Verjovsky Conjecture when the underlying manifold has solvable fundamental group. 
Ghys~\cite{Ghy89} proved this conjecture when\footnote{For convenience \( \bE^{su} = \bE^{s}\oplus \bE^{u}\) denotes the stable-unstable sub-bundle,  \( \bE^{cs} = \bE^{c}\oplus \bE^{s}\) denotes the centre-stable sub-bundle and similarly \( \bE^{cu} = \bE^{c}\oplus \bE^{u}\) denotes the centre-unstable sub-bundle.} $\bE^{su}$ is \(\cC^1\) or when the codimension-one sub-bundle is $\cC^{2}$ and the flow is volume preserving (i.e., $\dim \bE^{u}=1$ and $\bE^{cs}$ is $\cC^{2}$ or with stable and unstable swapped).  
These results were improved by Simi\'c~\cite{Sim96} who proved the conjecture under the additional assumption that \(\bE^{su}\) is Lipschitz and then~\cite{Sim97}  for 
the case where the codimension-one sub-bundle is $\cC^{1+\alpha}$ for all \(\alpha<1\).
The complete resolution of the conjecture was announced by Asaoka~\cite{Asa08} but a gap was found in a result of Simi\'c~\cite{Sim05} that had been used 
 in the proof \cite[Erratum]{Asa08}. In this article, using many of the ideas developed by Simić, we complete this work.

\begin{mainthm}
\label{thm:istrue}
 The Verjovsky Conjecture is true.
\end{mainthm}
\noindent
In order to prove the above result we first prove a result related to integrability of the \(\bE^{su}\) sub-bundle.
A sub-bundle is said to be \emph{integrable} if there exists a foliation tangent to it.
The integrability (or non-integrability) of sub-bundles in important from a dynamical systems point of view in many situations. 
The sub-bundles  \(\bE^s\) and \(\bE^u\) are both integrable but in general the sub-bundle \(\bE^{su}\) is not \emph{integrable}. 
In particular the set of Anosov flows where the stable and unstable bundles are  not jointly integrable is $\cC^{1}$-open and $\cC^{k}$-dense (for all \(k\in \bN\)) in the set of all Anosov flows~\cite{FMT07} (see references within concerning the prior work of Brin).
Nevertheless we are able to prove integrability of \(\bE^{su}\)  for certain exceptional Anosov flows. 

\begin{mainthm}\label{thm:main}
Let $F:=\{f^{t}, t\in\mathbb{R}\}:\cM \to \cM$ be a\footnote{We say that a flow is \(\cC^{1}\) if the associated vector field is \(\cC^{1}\).} 
 $\cC^{2}$ codimension-one volume preserving Anosov flow  with $\dim(M)>3$. If $F$ is \textit{specially synchronisable} then it is topologically equivalent to an Anosov flow whose stable and unstable are jointly integrable.
\end{mainthm}
\noindent
One of the main difficulties of integrating these bundles is lack of regularity and typically the sub-bundles of Anosov flows are merely H\"older. 
In our setting $\bE^{s}$ and $\bE^{cu}$ are $\cC^{1 }$ but it cannot be hoped in general that $\bE^{u}$ is better than H\"older.
The ideas of the proof of the above theorem are inspired by techniques developed by Luzzatto, T\"ureli and the  author~\cite{LTW16} for integrating bundles that are just continuous. 
 
 Section \ref{sec:Anosovflow} is devoted to recall general facts about Anosov flows and their reparametrisations. In this section, we will recall a bunching condition that guarantee a certain regularity of the stable bundle which is used to define the Parry's synchronisation. This is the parametrisation so that the unstable Jacobian is constant. We believe that this passage to the Parry's synchronisation is not necessary for the proof of the conjecture but it simplifies the calculation for constructing the adequate time change for integrability.

In Section \ref{sec:Spec}, we give the definition of special synchronisation for codimension-one Anosov flow. In Section \ref{sec:Int}, we discuss the integrability of continuous bundles and the existence of global cross section. In Section \ref{sec:proofmain} we prove Theorem \ref{thm:main}. Section \ref{sec:time} is devoted to prove that every codimension-one Anosov flow on a manifold of dimension at least four is specially synchronisable.

\begin{rem}\normalfont
  That the Vejovsky conjecture is true implies that codimension-one Anosov flows on higher dimensional manifolds only exist on manifolds with solvable fundamental group.
  Consequently Plante's result~\cite{Pla81} covered all cases but we do not know how to directly show this and take advantage of the previous work.  
\end{rem}

Combining Theorem~\ref{thm:main} with various known results allows us to prove the Theorem~\ref{thm:istrue}.
First we recall the following two results.

\begin{thmm}[Verjovsky {\cite[Theorem 1.1]{Ver74}}]\label{Ver74}
Every codimension-one Anosov flow on a closed manifold of dimension greater than three\footnote{  As shown by Franks \& Williams~\cite{FW80}, in the case where both the stable  and unstable bundles are at  least two dimensional there are examples of non-transitive Anosov flows. 
 In the three dimensional case the question of transitivity  of Anosov flows remains open.} is topologically transitive.
\end{thmm}

\begin{thmm}[Asaoka {\cite[Main Theorem]{Asa08}}]\label{Asa08}
 Every topologically transitive codimension-one Anosov flow is topologically equivalent to a \(\cC^{\infty}\) volume-preserving Anosov flow.
\end{thmm}
\noindent
Combining the two theorems means that it suffices to consider only $\cC^{\infty}$ volume-preserving Anosov flows.
When talking about  codimension-one  Anosov flow, without loss of generality we suppose \(\dim \bE^s =1\). 
We apply the process of synchronisation, as used by Parry~\cite{Par86}, which involves a time change of the flow. Full details are contained in Section~\ref{sec:Sim} and the pertinent details described in Proposition~\ref{prop:synchronisation}. In particular in Proposition \ref{prop:spesync}

\vspace{.8em}
{\em
\noindent
Suppose that a $\cC^{2 }$ volume-preserving Anosov flow with \(\dim \bE^s =1\) and \(\dim \bE^u \geq2\) is given.
Then the Parry's synchronized flow is topologically conjugate to the original flow and  satisfies the assumptions of Theorem~\ref{thm:main}. }
\vspace{.8em}

\noindent
Consequently, using Theorem~\ref{thm:main}, we conclude that every codimension-one Anosov flow on a manifold of 
dimension greater than three is topologically equivalent to a codimension-one Anosov flow with integrable \(\bE^{su}\). 
Plante proved the following, using integrability to demonstrate the existence of a global cross-section and the  work of Franks \cite{Fra70} and Newhouse \cite{New70} concerning the classification of Anosov diffeomorphisms.
\begin{thmm}[Plante {\cite[Theorem 3.7]{Pla72}}]\label{Pla72}
Every codimension-one Anosov flow with integrable \(\bE^{su}\) is topologically conjugate to a suspension of a hyperbolic toral automorphism.
\end{thmm}
\noindent
Since we have shown that it suffices to consider flows with 
 integrable \(\bE^{su}\) the proof of Theorem~\ref{thm:istrue}  
follows from the above result.

\section{Anosov flows and reparametrisation}\label{sec:Anosovflow}

\subsection{Reparametrisation of Anosov flow and Parry's formula}
In this section we review classical results about reparametrisation of Anosov flows. 
Let $F:=\{f^{t}, t\in\mathbb{R}\}: M\to M$ be an Anosov flow generated by a $C^1$ vector field $Z$. Let $\bE^s\oplus\bE^u$ be the corresponding hyperbolic splitting, i.e. there exsits $\lambda_s\in(0,1), \lambda_u>1$ and  $C>0$ such that for every $t>0$ we have
\begin{equation}\label{eq:Anosov1}
\|Df^t|_{\bE^s}\|\leq C\lambda_s^{-t}\quad\text{ and }\quad \|Df^{-t}|_{\bE^u}\|\leq C\lambda_u^{-t}
\end{equation}
where $\|.\|$ is the norm given by the Riemannian structure.

The definition of Anosov flow implies the existence of an invariant  $1$-form $\eta$ on $TM$, i.e.\footnote{ This one form can be explicitely defined by $\ker(\eta)=\bE^s\oplus\bE^u$ and $\alpha(Z)=1$. The invariance of $\ker(\eta)$ together with $\eta(Z)=1$ implies that $(f^{t})^*\eta=\eta$ $\forall t\in\mathbb{R}.$ }

\begin{equation}\label{eq:in1form}
(f^{t})^*\eta=\eta\quad\text{ and }\quad \eta(Z)\equiv1\quad \forall t\in\mathbb{R}.
\end{equation}
The codimension-one condition implies the definition of $1$-form $\omega$ such that 
\[
\ker(\omega)=\bE^{cu}:=\bE^u\oplus<Z>.
\]
In the next section we use a specific vector field that generates $X$ that generates the stable bundle and imposes the condition 
\[
\omega(X)\equiv1.
\]
so that $\omega$ is well defined. It is standard that if $\alpha_0$ is a $1$-form such that $\alpha_0(Z)\equiv1$ then the $1$-form 
\[
\alpha=\alpha_0-\alpha_0(X)\omega
\]
satisfies the following: there exists $C_1>0,\zeta>0$ such that for all $t>0$
\begin{equation}\label{eq:conv}
\|(f^{-t})^*\alpha-\eta\|\leq C_1e^{-\zeta t}.
\end{equation} 
In other words, the family of $1$-forms defines a dynamical approximation of $\eta$. We emphasize that in all the situations that this approximation will be used;  the vector field $X$, the $1$-forms $\omega$ and $\alpha_0$ are $C^1$ which will imply that $\{\alpha^{(t)}:=(f^{-t})^*\alpha, t>0\}$ is the $C^1$ approximation of $\eta$.

Let $\psi: M\to(0,\infty)$ be a $C^1$ function. It is standard that  if a flow generated by a vector field $Z$ leaves invariant a smooth volume form $m$ then the flow generated by $Z/\psi$ leaves invariant the volume form $\psi m$. In the case of Anosov flow, Anosov \cite{Ano69} proves that if $F=\{f^{t}, t\in\mathbb{R}\}$ is an Anosov flow generated by $Z$ then the flow $F_\psi:=\{f^{t}_\psi, t\in\mathbb{R}\}$ generated by $Z/\psi$ is also an Anosov flow. Furthermore, Parry \cite[\S6]{Par86} gives an explicite formulation of the associated invariant bunldes. More precisely, if $\bE_{\psi}^{s}\oplus\bE^{u}_{\psi}$ are the corresponding invariant splitting of $F_{\psi}$ we have
$$
\bE^{\sigma}_\psi:=\{Y+\xi^{\sigma}(Y)Z: Y\in\bE^\sigma\}, \sigma=s,u
$$
where
\begin{equation}\label{eq:xi}
\xi^{s}(Y)=\frac{1}{\psi}\int_0^{\infty}d(\psi\circ f^{t})(Y)dt \quad\text{ and }\quad \xi^{u}(Y)=\frac{1}{\psi}\int_0^{\infty}d(\psi\circ f^{-t})(Y)dt.
\end{equation}

One of the main ingredient in proving the conjecture is based on the following buchning condition: If $F$ is a volume preserving $C^2$ Anosov flow with $\dim(\bE^s)=1$ and $\dim(\bE^u)\geq2$ then there exists $\zeta, C>0$ such that:
\begin{equation}\label{eq:bunching}
\|Df^t|_{\bE^s}\|\cdot\|Df^t|_{\bE^u}\|\leq Ce^{-\zeta t}, \quad\text{ for all }\quad t>0.
\end{equation}
This bunching condition was initially observed by Plane in \cite{Pla72} and was recently used by the author and Butterley \cite{BW19} to construct  open sets of exponentially mixing Anosov flows. 

\subsection{Parry's synchronisation}\label{sec:Sim}

Suppose that $\widetilde F=\{\widetilde f^t: t\in\mathbb{R}\}:\cM \to \cM$ is a $\cC^{2 }$ volume-preserving Anosov flow with one-dimensional stable bundle 
and unstable bundle at least two-dimensional. 
The $\cC^{r}$ section theorem of Hirsh, Pugh and Shub~\cite{HPS77} implies that the unstable bundle $\tilde\bE^s$ is $\cC^{1 }$ because of the bunching between $\tilde\bE^s$ and $\tilde\bE^{cu}$. 
 The $\cC^{r}$ section theorem also means that the centre stable bundle $\tilde\bE^{cs}$ is $\cC^{1 }$ (since this sub-bundle is codimension-one). 
We may assume,\footnote{Simi\'c \cite{Sim97} also has an argument like this but we want to ensure better smoothness in our setting. We guarantee immediate expansion / contraction but do not make the stable and unstable bundles orthogonal and hence maintain the regularity of the metric.} by a smooth change of metric~\cite[Appendix~A]{BL13} if required, that the flow is immediately expanding (resp.\@ contracting) in the sense that there exists $\gamma>0$ such that $\norm{\smash{D{\tilde f}^{t} |{\tilde\bE^s_{q}}}}\leq e^{-\gamma t}$ and $\norm{\smash{D{\tilde f}^{-t}|{\tilde\bE^u_{q}}}}\leq e^{-\gamma t}$ for all $t\geq 0$, \(q\in \cM\).
Synchronisation, exactly as below, was used by Parry~\cite{Par86} in order to modify an Anosov flow in such a way that the SRB measure coincides with the measure of maximal entropy. It was also used by Simi\'c~\cite{Sim97} to prove the Verjovsky conjecture with additional regularity assumptions.

We denote by $\jac(D{\tilde f}^{t}| \tilde\bE^s)$  the unstable Jacobian. %
Let
\[
 u(q) := \lim_{t\to 0} \tfrac{1}{t} \log \jac(D{\tilde f}^{t}| \tilde\bE^s)(q)
 \]
and observe that  $u$ is $\cC^{1 }$ and negative. 
  We define the new vector field 
\[
  Z=\frac{1}{u}\tilde Z
\]
and consider \(F=\{f^t:t\in\mathbb{R}\}: \cM \to \cM\), the flow associated to $ Z$.

\begin{prop}\label{prop:synchronisation}
Suppose that $\widetilde F=\{\tilde f^t:t\in\mathbb{R}\}: \cM \to \cM$ is a $\cC^{2 }$ volume-preserving Anosov flow with \(\dim \tilde\bE^s =1\) and \(\dim \tilde\bE^u \geq2\).
Let  \(F=\{f^t:t\in\mathbb{R}\}: \cM \to \cM\) denote the \emph{synchronised} flow defined as above. Then
\begin{enumerate}
\item
\(F\) is a $\cC^{1 }$ Anosov flow;
\item
\(F\) is topologically equivalent to \(\widetilde F\);
\item
There exists a  non-zero $\cC^{1}$ vector field $X$ such that,  
\[
Df^{t} X = e^{-t}X,
\quad \text{for all \( t\in \bR\)}.
\] 
\end{enumerate}
\end{prop}
\noindent
The first two claims of the proposition are simple and as described by Parry~\cite{Par86}.
 For the convenience of the reader, the remainder of this section is devoted to collecting together the necessary details for the proof of the above.
 The existence of the vector field $X$ with the appropriate properties was proven by Simi\'c~\cite[\S4]{Sim97}.
 The required regularity of $X$ was claimed in the unpublished manuscript of Simi\'c~\cite[Lemma~3.3 of ver.\@ 3]{Sim05} but for completeness we include here the proof. 
\begin{rem}\normalfont
The existence of such a vector field $X$, as given in Proposition~\ref{prop:synchronisation},  is a strong property which tells us a lot about the global behaviour of the flow. 
Such a vector field, if it exists, is unique up to multiplication by a constant.
\end{rem}
The {synchronised} flow is a $\cC^{1 }$ Anosov flow because ${1}/{u}$ is $\cC^{1 }$.
Since the new flow is obtained by modifying only the speed of the flow, the new flow has exactly the same orbits as the original, i.e., it is topologically equivalent to the original flow. 
The central stable and central unstable bundles for the synchronised flow are identical to the bundles prior to synchronisation. 
However there is no reason to expect the new unstable bundle $ \bE^s$ to coincide with $\tilde\bE^s$ (and similarly for the stable bundle). 

In the remainder of this section we describe the proof of the final claim of Proposition~\ref{prop:synchronisation}. Without loss of generality we suppose that $\tilde\bE^{s}$ is orientable (if not then we pass to a $2$-cover of $\cM$). 
Fix $\tilde X$ a unit vector field tangent to $\tilde\bE^{s}$. 
Let $\omega$ be the one-form defined by requiring
\[
\ker \omega = \tilde\bE^{cu},
\quad
\omega(\tilde X)=1.
\]
Since the two sub-bundles used to define $\omega$ are both $\cC^{1 }$ we know that $\omega$ is $\cC^{1 }$.
By Frobenius, since $\tilde\bE^{cu}$ is integrable, there exists a continuous $1$-form $\eta$ such that
\[
\extd \omega = \eta \wedge \omega.
\]
\begin{lem}
\label{lem:goodeta}
Let $\eta$, $\tilde Z$ and $u$ be defined as above. Then $\eta(\tilde Z) = u$.
\end{lem}
\begin{proof}
Since $\omega(\tilde X) = 1$ and $\tilde Z \in \ker \omega$,  
\[
\begin{aligned}
\eta(\tilde Z) &= 
\eta(\tilde Z)\omega(\tilde X)    - \eta(\tilde X) \omega(\tilde Z) \\
&= (\eta \wedge \omega) (\tilde Z, \tilde X)\\
&= \extd \omega (\tilde Z, \tilde X).
\end{aligned}
\]
 Cartan's formula in this setting reads $\cL_{\tilde Z}\omega = i_{\tilde Z} \extd \omega + \extd (i_{\tilde Z}\omega)$. Since $\tilde Z \in \ker \omega$ this means that $  \extd \omega (\tilde Z, \tilde X) = (\cL_{\tilde Z}\omega)(\tilde X)$.
Consequently $\eta(\tilde Z) =  (\cL_{\tilde Z}\omega)(\tilde X)$.
Observe that 
$\pb{{\tilde f}^{t}}\omega
= \jac(D\tilde f^t|{\tilde\bE^s}) \ \omega$.
Since
$
\cL_{\tilde Z}\omega
=
\left.\frac{d}{dt}\right|_{t=0} \pb{\tilde f^{t}}\omega$
it follows that $\eta(\tilde Z) = (\cL_{\tilde Z}\omega)(\tilde X) = u$.
\end{proof}

Let $X$ be defined as the vector field tangent to $\bE^{s}$ such that $\omega(X) =1$ (subsequently we will obtain an exact formula for $X$).
\begin{lem}
\label{lem:goodX}
Let $F$ be the synchronised flow and let $X$ be the vector field defined above. 
Then $\dm{f^{t}}X = e^{-t}X$ for all $t$.
\end{lem}
\begin{proof}
Observe that $\eta(Z) = \eta(\frac{1}{u}\tilde Z) = 1$ by Lemma~\ref{lem:goodeta}.
Since $\bE^{cu} = \tilde\bE^{cu}$ we know that $Z \in \ker \omega$ and by choice  $\omega(X) =1$.
Therefore, similar to above,
\[
\begin{aligned}
\eta( Z) &= 
\eta( Z)\omega( X)    - \eta( X) \omega( Z) \\
&= (\eta \wedge \omega) ( Z,  X)\\
&= \extd \omega ( Z,  X)\\
&=  (\cL_{ Z}\omega)( X).
\end{aligned}
\]
That $(\cL_{ Z}\omega)( X) = 1$ implies that, as required, $ \dm{f^{t}}X = e^{-t}X$ for all $t$.
\end{proof}

To finish the proof of Proposition~\ref{prop:synchronisation} it would suffice to show that $\bE^u$   is a $\cC^{1}$ sub-bundle.
Unfortunately this doesn't follow  from the $\cC^{r}$ section theorem
~\cite{HPS77}   because the synchronised flow is merely $\cC^{1 }$ and so we need to argue directly and use the connection to the original flow.

\begin{rem}\normalfont
 In this proof we are not using the full strength of the bunching~\eqref{eq:bunching} which we have in this setting.
 If required, using the formulae given below, it can be shown that $\bE^u$   is a $\cC^{1+\alpha}$ sub-bundle for some \(\alpha>0\).
 We don't pursue this direction because it is not required for our present purposes.
\end{rem}

\begin{lem}
\label{lem:regularity}
Let $X$ be as  defined above. Then $X$ is a $\cC^{1}$ vector field.
\end{lem}
\begin{proof}
We have that \cite[\S6]{Par86}
\[
X= \tilde X +\xi(\tilde X) \tilde Z 
\quad\text{ where }\quad 
 \xi(\tilde X)
 = \frac{1}{u}\int_0^{\infty}\tilde X(u\circ \tilde f^{t})
 \ dt.
\]
The above integral is well defined since \( 
\norm{\smash{\dm{{\tilde f}^{t}}| {\tilde\bE^s}}} \leq e^{-\gamma t} \) and 
the above identity is verified by showing that the bundle defined as above is indeed invariant under the action of the synchronised flow ${\flow{Z}{t}}$ \cite[\S6]{Par86}. 
We merely need to show that $x\mapsto\xi_{x}(\tilde X)$ is \(\cC^{1}\).
As already noted \(\jac(\dm{{\tilde f}^{t}}| \tilde\bE^s)\) is  \(\cC^{1 }\). 
Recall that we are working in a metric where $\tilde\bE^{cu}$ and $\tilde\bE^{s}$ are orthogonal but which agrees with the smooth original metric within the sub-bundles.
This means that \(u\) is as smooth as \(\dm{\flow{\tilde Z}{t}}\) along the unstable bundle and so \(\tilde X(u)\circ {\tilde f}^{t}\) is \(\cC^{1 }\).
To complete the proof of the regularity we take advantage of the bunching~\eqref{eq:bunching} between $\tilde\bE^{cs}$ and $\tilde\bE^{u}$. 

First observe that
\[
 \tilde X(u\circ\flow{Z}{-t})(q) =  \jac(\dm{\tilde f^{t}}| \tilde\bE^s)(q) \cdot  (\tilde Xu)\circ \tilde f^{t}(q).
\]
For convenience let 
\(\lambda_t(q):= \jac(\dm{\tilde f^{t}} | \tilde \bE^u)\).
Differentiating we obtain sum of two terms, 
\begin{equation}
 \label{eq:twoterms}
D\lambda_t(q)   (\tilde Xu)( \tilde f^{t} q)
+
\lambda_t(q)   (D(\tilde Xu))( \tilde f^{t} q) \dm{\tilde f^{t}}(q). 
\end{equation}
We will estimate the $\cC^{1}$ norm of these two terms.
Suppose first that \(t\in \bR\) and note that
$\lambda_t(q) = \prod_{s=0}^{t-1} \lambda(\tilde f^{s}q)$
where, for convenience, we write $\lambda = \lambda_1$.
Consequently
\[
 D\lambda_t(q) = \sum_{s=0}^{t-1} 
    \frac{ D\lambda}{\lambda} (\tilde f^{s}q) D\tilde f^{s}(q) \lambda_t(q).
\]
Bunching means that there exists $\zeta>0$, $C>0$ such that, for all $s\geq 0$, $q \in \cM$,
\(
 \norm{\smash{D\tilde f^{s}(q)}} \lambda_{s}(q) \leq Ce^{\zeta s}
\).
We also observe that
\(\sup_{q\in \cM}  \frac{ \norm{D\lambda}}{\lambda} (q) <\infty\)
and 
\(\frac{\lambda_t(q) }{\lambda_s(q)} = \lambda_{t-s}(\flow{Z}{s}q) \leq e^{-\eta(t-s)}\) for all \(0\leq s \leq t\).
Consequently (increasing \(C>0\) as required)
\[
\begin{aligned}
 \norm{D\lambda_t(q)} 
 &\leq
 \left( \sup_{y\in \cM}  \frac{ \norm{D\lambda}}{\lambda} (y) \right) \sum_{s=0}^{t-1} 
   \norm{\smash{D\tilde f^{s}(q)}}  {\lambda_t(q) }   \\
 &\leq
 C  \sum_{s=0}^{t-1} 
    e^{-\zeta s} \frac{\lambda_t(q) }{\lambda_s(q)}
    \leq
    C e^{-\eta t} 
     \sum_{s=0}^{t-1} 
e^{s(\eta-\zeta)}
= C e^{-\eta t}(e^{t(\eta-\zeta)} - 1).
    \end{aligned}
\]
This estimate, exponentially small with increasing \(t\), holds also for all \(t\in \bR\) by increasing \(C>0\).
For the second term of~\eqref{eq:twoterms} we use the easy estimate
\[
 \lambda_t(q) \norm{\smash{ \dm{\tilde f^{t}}(q)}}
 \leq Ce^{\zeta t}
\]
from the bunching property.
\end{proof}

 The idea of proving Theorem is about introducing a new time change of the synchronized Anosov flow to guarantee the joint integrability of stable and unstable bundle. To this end will need to show the following

\begin{prop}\label{prop:c1s}
Let $F=\{f^t\}$ be a codimension-one Anosov flow that is synchronized as in Proposition \ref{prop:synchronisation}. If $\psi: M\to(0,\infty)$ is a $C^2$ function then the vector field $Z_\psi=Z/\psi$ generates an Anosov flow whose stable bundle is $C^1$.
\end{prop}
\begin{rem}
Observe that Proposition \ref{prop:c1s} would be given by the $C^1$ section Theorem is the flow was $ C^2$, however the synchronized flow is just $C^{1+\theta}$. 
\end{rem}
\begin{proof}The fact that the flow generated by $Z_\psi$ is Anosov follows directly from \ref{Ano69}. To prove that $\bE^s_\psi$ is $C^1$ we will use its  formula. By \eqref{eq:xi}, $\bE^s_\psi$ is generated by the vector field
\[
X_\psi=X+h_\psi Z
\quad
\text{ where } 
\quad
h_\psi=\frac{1}{\psi}\int_0^{\infty}e^{-t}X(\psi)\circ f^t dt.
\]
Therefore to prove that $\bE^s_\psi$ is $C^1$, it is enough to prove that $g_\psi$ is $C^1$. It is easy to see that for any vector field $V\in TM$ we have
\[
\|V(g_\psi)\|\leq \int_0^{\infty}e^{-t}\|Df^t\|\|\psi\|_{C^2}dt
\]
Thus since the synchronized flow satisfies the bunching condition, we have the the integral converges therefore $g_\psi$ is $C^1$.
\end{proof}

\section{Special synchronisation and Integrability}

\subsection{Special synchronisation}\label{sec:Spec}
The aim of this section is to prove Theorem \ref{thm:main} under the special synchronisation condition. throughout this section, we suppose that  we have a $C^1$ codimension-one volume preserving Anosov flow $F=\{f^t,t\in\mathbb{R}\}: M\to M$. 
We also suppose that the stable bundle, which is one dimensional, is spanned by a $C^1$ vector field $X$.

\paragraph{Admissible disk:}
An admissible $\mathcal D$ disk is a codimension-one open $C^1$ submanifold that is transverse to the flow $F$, tangent to $X$ and contains an unstable manifold of $F$.  More precisely, if $\mathcal L_p$ is a local unstable manifold through $p\in M$ and $\varepsilon>0$, a typical admissible disk through $p$ is of the form
\[
\mathcal D_p:=\bigcup_{|s|<\varepsilon}e^{sX}\mathcal L_p
\]
where $e^{sX}$ denotes the time $s$ map of the flow generated by the vector field $X$ through $p.$

\begin{rem}\normalfont
We notice that if the vector field $X$ is $C^1$ then $\mathcal D_p$ defines a $C^1$ submanifold. It has the properties that $X$ is tangent to $\mathcal D_p$ and the local unstable manifold $\mathcal L_p$ is in $\mathcal D_p$. However this disk is  tangent to $\bE^s\oplus\bE^u$ only when the stable and unstable bundles are jointly integrable.
\end{rem}

\paragraph{Admissible section:} An admissible section for $F$ is a finite collection of admissible disk $\mathcal S:=\{\mathcal D_1,\cdots,\mathcal D_\ell\}$ for some $\ell\in\mathbb{N}$ such that:
\begin{enumerate}
\item for every $q\in M$, there exists $t_q^{\pm}\geq0$ such that $f^{t_q^+}(q)\in \mathcal S$ and $f^{-t_q^-}(q)\in \mathcal S$,
\item $\overline{\mathcal D}_i\bigcap\overline{\mathcal D}_j=\emptyset$ for $i\neq j$,
\end{enumerate}
where $\overline{\mathcal D}_i$ defines the closure of ${\mathcal D}_i$.

\paragraph{Admissible flow box:} Given an admissible disk $\mathcal D$ and $\tau>0$ an admissible flow box of length $\tau$ is defined by 
\[
\mathcal U^{(\tau)}:=\bigcup_{|t|<\tau}f^t\mathcal D.
\]

\paragraph{The space of functions:} 
Given a continuous vector field $Y$ on $M$ and a $C^1$ function $\psi$, we write $Y(\psi)=d\psi(Y)$ where $d\psi$ is the $1$-form given by the exterior derivative of the $0$-form \(\psi\). Let $\mathfrak D$ be the space of  functions $\psi$ such that:
\begin{enumerate}
\item $\psi, X(\psi)\in C^1$,
\item \({1}<\psi(q)<2\), for every $q\in M$.
\end{enumerate} 
If $\psi\in\mathfrak D$, we define
\begin{equation}
\|\psi\|_{\mathfrak D}=\max\{\|\psi\|_{C^0}, \|g_{\psi}\|_{C^0}, \sup_{V\in\bE^{cu}, \|V\|=1}\|V(\psi)\|_{C^0},  \sup_{V\in\bE^{u}, \|V\|=1}\|V(g_{\psi})\|_{C^0}\}
\end{equation}
\begin{rem}\normalfont
We notice that $\mathfrak D$ is not a linear space, let alone a Banach space. The motivation of the definition of $\|.\|_{\mathfrak D}$  is to allow the flexibility that the derivative in the stable direction is not controlled however all the quantities that are involve in the definition of the stable and unstable bundle of a reparametrized flow by a function in $\mathfrak D$ are controlled by $\|.\|_{\mathfrak D}$. Precisely the function $\xi$ that is defined in \eqref{eq:xi} is controlled by the norm $\|.\|_{\mathfrak D}$. In particular this space helps to define certain flow whose vector field is not necessarily $C^1$ but are Anosov in the sense of \eqref{eq:Anosov1} as we see in Proposition \ref{prop:Anosovweak}.
\end{rem}

\begin{prop}\label{prop:Anosovweak}
If $\{\psi_n, n>1\}\subset\mathfrak D$ is such that $\|\psi_n-\psi_{n+1}\|_{\mathfrak D}\leq \theta^n$ for some $\theta\in(0,1)$ then $\{\psi_n, n>1\}$ converges in the $C^0$ topology to a function $\psi^*$ such that the vector field $Z_{\psi^*}=Z/\psi^*$ generates a flow that is Anosov as in \eqref{eq:Anosov1}.
\end{prop}
\begin{proof}
For $n\geq1$, let $F_n:=\{f^{t}_n:t\in\mathbb{R}\}$ be the flow generated by the vector field $Z_n=Z/\psi_n$ and $X^{(n)}$ is the corresponding vector field that spans the stable bundle of $F_n$.
Since
\(
\|\psi_n-\psi_{n+1}\|_{\mathfrak D}\leq \theta^n
\)
for some $\theta\in(0,1)$ then, in particular, $\psi_n$ has a limit $\psi^*$ in the $C^0$ topology that is differentiable along the weak unstable bundle. Let $t\in\mathbb{R}$ we want to prove that, $Df^t_n$, the time $t$ differential of the flow generated by the vector field $Z_n:=Z/\psi_n$ converges uniformly in order to have that the limit flow is $C^1$.

Since the derivatives of $\psi_n$ along the center unstable bundle are converging uniformly then $Df^{t}_n|_{\bE^{cu}}$ converges uniformly. Therefore to prove that $Df^{t}_n$ converges uniformly, it is enough to prove that $Df^{t}_n|_{\bE^s}$ converges uniformly.  
\begin{claim}\label{claim1}We have the following
\[
Df^t_nX^{(n)}=\exp\left(\int_0^t\frac{u}{\psi_n}\circ f^{s}_nds\right)X^{(n)}
\]
where $u:=\frac{d}{dt}|_{t=0}\jac(D{ f}^{t}| \bE^s)$.
\end{claim}
\begin{proof}
By \eqref{eq:xi} we have
\[
X^{(n)}=X+g_{\psi_n}Z_n\quad\text{ where } \quad g_{\psi_n}=\int_0^{\infty}d(\psi\circ f^{t})(X)dt.
\]
Using linearity of the Lie bracket we have
\[
\begin{aligned}
\left[X^{(n)},Z_n\right]&=\left[X,\frac{Z}{\psi_n}\right]+\left[g_{\psi_n}\frac{Z}{\psi_n}, \frac{Z}{\psi_n}\right]\\
&=-\frac{X(\psi_n)}{\psi_n^2}Z+\frac{1}{\psi_n}\left[X,Z\right]-\frac{Z(g_{\psi_n})}{\psi_n^2}Z
\end{aligned}
\]
Since $Df^tX=\jac(D{ f}^{t}| \bE^s)X$ then we have $[X,Z]=uX$, similarly, it is easy to see that $Z(g_{\psi_n})=g_{\psi_n}+uX(\psi_n)$. Substituting this into the above gives
\[
\left[X^{(n)},Z_n\right]=\frac{u}{\psi_n}X^{(n)}
\]
therefore using that $\left[X^{(n)},Z_n\right]=\frac{d}{dt}|_{t=0}Df^t_nX^{(n)}$ then we get the desired equality.
\end{proof}
Since $X^{(n)}$ converges uniformly to a continuous vector field then we have $Df^{t}_nX^{(n)}$ converges uniformly. 
Thus $F^*$ defines a $C^1$ flow that is Anosov.
\end{proof}

The following Lemma gives  regularity of the stable foliation.

\begin{lem}\label{lem:DeX}
Let $\{\psi_n, n>1\}\subset\mathfrak D$  such that $\|\psi_n-\psi_{n+1}\|_{\mathfrak D}\leq \theta^n$ for some $\theta\in(0,1)$ and $X^{(n)}$ denotes the vector field that spans the stable bundle of $Z_n=Z/\psi_n$. Then
for every $s\in\mathbb{R}$, the sequence of differential of the flow given of $X^{(n)}$, $\{De^{sX^{(n)}}\},$ converges uniformly. In particular the $X^{*}$ defines a $C^1$ flow.
\end{lem}
\begin{proof}
To prove the lemma, it is enough to choose a set of continuous vector field and evaluate the differential on that set. It is easy to see that $De^{sX^{(n)}}X^{(n)}=X^{(n)}$ and since for $V\in\bE^u$ the functions $V(g_{\psi_n}) V(\psi_n)$ converge uniformly to  continuous functions then if $V\in\bE^u$ is a continuous vector field then $De^{sX^{(n)}}V$ converges to a continuous vector field. In order to conclude the proof, we only need to evaluate $De^{sX^{(n)}}Z_n$. As in the proof of Claim \ref{claim1}, we have
$[X^{(n)}, Z_n]=\frac{u}{\psi_n}Z_n$ therefore we have
\[
De^{sX^{(n)}}Z_n=Z_n+\int_0^s\frac{u}{\psi_n}\circ e^{\tau X^{(n)}}d\tau X^{(n)}
\]
Then we have $De^{sX^{(n)}}Z_n$ is converging uniformly. Thus $De^{sX^{(n)}}$ converges uniformly and $X^*$ defines a $C^1$ flow.
\end{proof}

\begin{definition}\normalfont[Special synchronisation]\label{def:sync}

The flow $F$ is said to be \textit{specially synchronisable} if:
\begin{enumerate}
\item there exists a $C^1$ differential $1$-form $\alpha$ such that 
\(X\in\ker(\alpha), \alpha(Z)\equiv1\) 
\item there exists a sequence of functions $\{\psi_n, n\geq1\}\subset\mathfrak D$ and $\theta\in(0,1)$ such that $\|\psi_n-\psi_{n+1}\|_{\mathfrak D}\leq \theta^n$

\item there exist an admissible section $\mathcal S$ for the flow $F$, a sequence of $T_n\to\infty$ and $\tau_n\to0$ with :
\begin{equation}\label{eq:tau}
\|Df^{T_n}|_{\bE^s}\|\leq \tau_n^2
\end{equation}
and
\begin{equation}\label{eq:synceq}
\|i_{X_{\psi_n}}\circ d\alpha^{(n)}_{f^{T_n}\mathcal S^{(\tau_n)}}\|\leq\theta^n 
\quad
\text{where} 
\quad
\alpha^{(n)}=\psi_n\alpha-h_{\psi_n}\omega.
\end{equation}
\end{enumerate}
\end{definition}
\begin{rem}\normalfont
We observe that the first item in the above definition is trivial as we will see in the next sections. The most important part of the definition is given by \eqref{eq:synceq}.
\end{rem}

In Section \ref{sec:proofmain}, we will prove Theorem \ref{thm:main} that states that special synchronisability implies joint integrablity of the limit Anosov flow.

\subsection{Integrability and global cross section}\label{sec:Int}
The problem of integrability of $C^1$ sub-bundles has been resolved by the well known Frobenius Theorem  which states that a smooth $1$-form $\eta$ is integrable if and only if it satisfies the involutivity condition
\begin{equation}\label{eq:inv}
\eta\wedge d\eta=0
\end{equation}
where $d\eta$ is the exterior derivative of $\eta$. In the case of a continuous $1$-form, the problem of integrability is more subtile in the sense that the exterior derivative $d\eta$ might not exist. However in the context of smooth dynamical systems most invariant $1$-forms are typically just H\"older continuous nevertheless there are dynamical methods to integrate them. For instance, the well known stable manifold theorem solve the problem of integrability of certain H\"older continuous $1$-forms. In \cite{LTW16}, we give a continuous version of the Frobenius Theorem that gives an alternative proof of the stable manifold Theorem. The ideas of integrability in \cite{LTW16} are the main inspiration of this work.

On the other hand, we notice that the existence of exterior derivative is a weaker condition than having smoothness of the $1$-form; indeed if $\psi$ is a $C^1$ function then the $1$-form $d\psi$ is just continuous however has an exterior derivative that is identically zero. Along the same lines, Hartman \cite{Hart82} defines the notion exterior derivative of a continuous $1$-form via Stokes Theorem.

\begin{definition}
A continuous $1$-form $\eta$ has  exterior derivative if there exists a continuous $2$-form $\beta$ such that 
$$
\int_J\eta=\int_D\beta
$$
for every Jordan curve $J$ that bounds a disk $D$. In this case we write $d\eta=\beta$.
\end{definition}
With this definition, Hartman \cite{Hart82} gives a similar statement to Frobenius Theorem:
Let $\eta$ be a continuous $1$-form that has a continuous exterior derivative. If $\eta$ is involutive then it is integrable (\cite[Theorem 3.1]{Hart82}). Weaker forms of Frobenius Theorem has been proved by the author in \cite{LTW16} that in particular gives an alternative proof of the well known stable manifold theorem for hyperbolic systems. However for the purpose of this paper, we will apply the above version of Hartman.

The case of integrability of $\bE^s\oplus\bE^u$ is very special, Plante \cite{Pla72} proves that following
\begin{prop}
Let $F=\{f^t, t\in\mathbb{R}\}:M\to M$ be an Anosov flow such that for every $q\in M$, there exists a $C^1$ local submanifold tangent to $\bE^s\oplus\bE^u$. Then corresponding invariant $1$-form $\eta$ that is defined in \eqref{eq:in1form} has an exterior derivative that is zero, i.e
\[
d\eta=0.
\]
\end{prop}

\begin{definition}
A flow $F=\{f^t, t\in\mathbb{R}\}$ has a global cross section if there exists a $C^1$ closed submanifold of codimension-one such 
\[
\forall q\in M, \exists t_q\in\mathbb{R}: f^{t_q}(q)\in N.
\]
\end{definition}
The following is also proved by Plante
\begin{thm}[Plante {\cite[Theorem 3.7]{Pla72}}]\label{prop:Plante}
Let $F=\{f^t, t\in\mathbb{R}\}:M\to M$ be a codimension-one Anosov flow such that for every $q\in M$, there exists a $C^1$ local submanifold tangent to $\bE^s\oplus\bE^u$. Then $F$ admits a global cross section.
\end{thm}

\subsection{Proof of Theorem \ref{thm:main}}\label{sec:proofmain}

Let $\{\psi_n, n>1\}\subset\mathfrak D$  be given by the definition of special synchronisability and $\psi^*$ be a $C^0$ limit of the sequence $\{\psi_n, n>1\}$. By Proposition \ref{prop:Anosovweak}, the vector field $Z^*=Z/\psi^*$ defines an Anosov flow $F^*$; let $\eta^*$ be the corresponding invariant $1$-form given in \eqref{eq:in1form}. The proof of Theorem \ref{thm:main} consists of proving that $\eta^*$ which concludes the proof of Theorem \ref{thm:main}. For notational purposes the flow generated by $Z_n:=Z/\psi_n$ is denoted by $F_n:=\{f^t_n, t\in\mathbb{R}\}$.

Let $\phi^\pm: M\to \mathbb{R}$ be the function defined by 
\[
\phi^{\pm}(q):=\min\{t>0: f^{\pm t}(q)\in\mathcal S\}
\]
This corresponds to the first return time in forward and backward time for $q$ to the section $\mathcal S$ with respect to the flow $F$. To simplify the notation, for $q\in M$, we write $\mathcal D^{\pm}(q)$ to mean the piece of section in which $q$ returns for the first time in forward and backward time, i.e.
\[
f^{\phi^{\pm}(q)}(q)\in\mathcal D^{\pm}(q).
\]
We recall that since the pieces of sections are $C^1$ then $\phi^{\pm}$ define piecewise $C^1$ functions and for $n>1$, we write
\[
\phi^{\pm}_n:=\phi^{\pm}\circ f^{-T_n}.
\]
Notice that
 the functions $\phi^{\pm}_n$ defines the first return times in forward and backward times with respect to the section $f^{T_n}\mathcal S$. Moreover since the local sections are tangent to the stable bundle of $F$ which is spanned by the vector field $X$ then we have
 \begin{equation}\label{eq:Xphi}
 X(\phi^{\pm}_n)_q=0
 \end{equation}
 where $q$ is a point of differentiability of $\phi^{\pm}_n$. Similarly, by the definition of Anosov flow, $Df^{-T_n}$ does not expand vectors in $\bE^{cu}$; there exists a constant $\Cl{con:Ano1}>0$ that only depends on $\mathcal S$ such that 
 \begin{equation}\label{eq:Yphi}
 |Y(\phi^{\pm}_n)_q|\leq\Cr{con:Ano1}\|Y\|
 \end{equation}
where $q$ is a point of differentiability of $\phi^{\pm}_n$ and $Y\in\bE^{cu}$.
We notice that the non-differentiable point of $\phi^\pm_n$ is given by the set $f^{T_n}\partial\mathcal S$; the iterates of the boundaries of the section under $f^{T_n}$.

For $n>1$, let $b_n^\pm: M\to\mathbb{R}$ be the function defined as follows: if $q\in M$ we consider $f^{T_n}\mathcal D^+(q)$ and $f^{T_n}\mathcal D^-(q)$ where be the piece of section whose orbit of $q$ returns for the first time in forward and backward time respectively. We define 
$$
W^{s,\pm}_{loc, n}(q):=\{q'\in W^{s}_n(q): f^{T_n}\mathcal D^\pm(q)=f^{T_n}\mathcal D^\pm(q')\}
$$
where $W^{s}_n(q)$ is the local stable manifold through $q$ under the flow $F_n$.
This consists of points on the stable manifolds of $q$ with respect to $F_n$ that share the the same first return section.
$b^\pm_n(q)$ is the length of the piece of stable $W^{s,\pm}_{loc, n}(q)$.

We define   $\overline\phi_n^\pm: M\to \mathbb{R}$ by
\[
\overline\phi_n^\pm(q)=\frac{1}{b_n^\pm}\int^{b_n^\pm}_{0}\phi^{\pm}_n\circ\gamma(t)dt
\]
where $\gamma$ is a unit speed parametrisation of the piece of stable manifold $W^{s,\pm}_{loc, n}(q)$.
We observe that, from its definition, $\overline\phi^{\pm}_n$ is piecewise constant along stable manifolds of $F_n$ and defines a piecewise $C^1$ function. 

\begin{lem}\label{lem:retu}
For every $q\in M$ and $n$ large enough  we have 
\[
f_n^{\overline\phi_n^\pm(q)}(q)\in\mathcal U^{(\tau_n)}.
\]
\end{lem}
\begin{proof}
For $q'\in W^{s,\pm}_{loc, n}(q)$ we have
\[
\left|\overline\phi_n^\pm(q)-\phi_n^\pm(q')\right|=\left|\frac{1}{b_n^\pm}\int_{W^{s,\pm}_{loc, n}(q)}\phi^{\pm}_n\circ\gamma(t)-\phi^{\pm}_n(q')dt\right|\leq \|X^{(n)}(\phi^{\pm}_n)\|b_n^{\pm}
\]
where the last inequality uses the Mean Value Theorem. Using \eqref{eq:Xphi}, \eqref{eq:Yphi} and the formula of $X^{(n)}$, we have $\|X^{(n)}(\phi^{\pm}_n)\|\leq C$ where $C>0$ is a constant that is estimated by  $C\leq\sup_{n}\|\psi_n\|_{\mathfrak D}<\infty$. Since the width of every piece of section $f^{T_n}\mathcal D_i$ is bounded by $\|Df^{T_n}|_{\bE^s}\|$ then using the formula of $X^{(n)}$ we have $b_n^\pm\leq C\|Df^{T_n}|_{\bE^s}\|$. Therefore using \eqref{eq:tau} in the definition of special synchronisability we have
\[
\left|\overline\phi_n^\pm(q)-\phi_n^\pm(q')\right|\leq C\tau_n^2.
\]
Since the constant $C$ does not depend on $n$ then for $n$ large enough we have $\left|\overline\phi_n^\pm(q)-\phi_n^\pm(q')\right|<\tau_n$ which gives the lemma.
\end{proof}
From the definition of special synchronisability, we have $ \|i_{X^{(n)}}\circ d\alpha^{(n)}_{f^{T_n}\mathcal U^{(\tau_n)}}\|\leq\theta^n
$, therefore we can choose a sequence  $t_n\to\infty$ such that  
\begin{equation}\label{eq:tn}
\|Df^{-t_n}_n\|\cdot \|i_{X^{(n)}}\circ d\alpha^{(n)}_{f^{T_n}\mathcal U^{(\tau_n)}}\|\leq\theta^{n/2}.
\end{equation}
We let  $\eta^{(n)}$ be the $1$-form defined by 
\[
\eta^{(n)}:=\varphi_n^+\cdot(f_n^{-t_n-\overline\phi^+_n})^*\alpha^{(n)}+\varphi_n^-\cdot(f_n^{-t_n-\overline\phi^-_n})^*\alpha^{(n)}.
\]
where
\[
\varphi_n^\pm:=\frac{\overline\phi_n^\pm}{\overline\phi_n^++\overline\phi_n^-}.
\]
We recall that since the functions $\overline\phi^{\pm}$ and $\overline\phi^{\pm}_n$ are piecewise differentiable then the $1$-form $\eta^{(n)}$ is also piecewise differentiable.

\begin{proof}[Proof of Theorem \ref{thm:main}]

First of all we recall that from the definition of special synchronisation, $\psi_n\to\psi^{*}$ in $\mathfrak D$. In particular, by the definition of $\|\psi_n-\psi^*\|_{\mathfrak D}$ if $\{Y_1,\cdots, Y_m\}$ span $\bE^u_q$ for some $q$ and 
\[
Y_i^{(n)}:=Y_i+\frac{1}{\psi_n}\int_0^{\infty}d(\psi_n\circ f^{-t})(Y_i)dtZ,\quad i=1,\cdots m
\]
spans the unstable bundle of $Z_n$ and converges to the unstable bundle of $Z_{*}$. Similalry $X^{(n)}$ converges to $X^*$ the stable bundle of $Z_{*}$.
Since $X^{(n)}\in\ker(\eta^{(n)})$ and $\eta^{(n)}(Z_n)\equiv1$ then using Cartan formula we have
\begin{equation}\label{eq:inv1}
d\eta^{(n)}(X^{(n)},Z_n)_q\equiv0
\end{equation}
where $q\in M$ is a point where $\eta^{(n)}$ is differentiable. It is easy to see that 
\begin{equation}\label{eq:inv0}
\begin{aligned}
d\eta^{(n)}=&d\varphi_n^+\wedge (f_n^{-t_n-\phi^+})^*\alpha^{(n)}+\varphi_n^+\left(d\overline\phi^+_n\wedge (f_n^{-t_n-\phi^+})^*\mathfrak{L}_{Z_n}\alpha^{(n)}+ \overline\phi^+_n(f_n^{-t_n-\phi^+})^*d\alpha^{(n)}\right)\\
&+d\varphi_n^-\wedge (f_n^{-t_n-\phi^-})^*\alpha^{(n)}+\varphi_n^-\left(d\overline\phi^-_n\wedge (f_n^{-t_n-\phi^-})^*\mathfrak{L}_{Z_n}\alpha^{(n)}+ \overline\phi^-_n(f_n^{-t_n-\phi^-})^*d\alpha^{(n)}\right).
\end{aligned}
\end{equation}
We will estimate each term separately.  For the first  and fourth we have
\begin{equation}\label{eq:inv2}
d\varphi_n^\pm\wedge (f_n^{-t_n-\phi^\pm})^*\alpha^{(n)}(X^{(n)},Y^{(n)}_i)=X^{(n)}(\varphi_n^\pm)\cdot(f_n^{-t_n-\phi^\pm})^*\alpha^{(n)}(Y^{(n)}_i)=0
\end{equation}
where the last equality uses $X(\varphi_n^+)=0$.
For the second and fifth term we have
\begin{equation}\label{eq:inv3}
d\overline\phi^\pm_n\circ f^{-t_n}\wedge (f_n^{-t_n-\phi^\pm})^*\mathfrak{L}_{Z_n}\alpha^{(n)}(X^{(n)},Y_i^{(n)})=X^{(n)}(\overline\phi^\pm_n\circ f^{-t_n})(f_n^{-t_n-\phi^\pm})^*\mathfrak{L}_{Z_n}\alpha^{(n)}(Y_i^{(n)})=0
\end{equation}
where the last equality uses that, from its definition, we have $X^{(n)}(\overline\phi^\pm_n\circ f^{-t_n})=0$. For the third an sixth term, we use the definition of special synchronisation and Lemma \ref{lem:retu} to have 

\begin{equation}\label{eq:inv4}
\left|(f_n^{-t_n-\phi^\pm})^*d\alpha^{(n)}(X^{(n)},Y^{(n)}_i)\right|\leq C\|Df^{-t_n}_n\|\cdot \|i_{X^{(n)}}\circ d\alpha^{(n)}_{f^{T_n}\mathcal U^{(\tau_n)}}\|
\end{equation}
Substituting , \eqref{eq:inv2}, \eqref{eq:inv3} and  \eqref{eq:inv4} into \eqref{eq:inv0} and using \eqref{eq:inv1} and \eqref{eq:tn} we have
\begin{equation}
\|i_{X^{(n)}}\circ d\eta^{(n)}\|\leq C\theta^{n/2}
\end{equation}
The rest of the proof consist of constructing a submanifold that is tangent to $\ker(\eta^*)$ through every point. To this end, we take an unstable manifold $\mathcal L_*$ of $Z_*$ that is parametrized by a $C^1$ map $\phi: (-\varepsilon,\varepsilon)^m\to M$ for some $\varepsilon>0$. We consider the following object $\mathcal W_n: (-\varepsilon,\varepsilon)^{m+1}\to M$:
\[
\mathcal W_n(s,u_1,\cdots, u_m):=e^{sX^{(n)}}\circ \phi(u_1,\cdots,u_m).
\]
\begin{claim} For $s\in(-\varepsilon,\varepsilon)$, 
if $\gamma_s:(-\varepsilon,\varepsilon)\to M$ is the curve defined by 
$$
\gamma_{s,i}(u):=e^{sX^{(n)}}\circ \phi(u_1,\cdots, u_{i-1},u, u_{i+1},\cdots,,u_m)
$$
then 
\[
\lim_{n\to\infty}\int_{\gamma_{s,i}}\eta^*=0.
\]
\end{claim}
\begin{proof}
Let $\Gamma=\gamma_{0,i}\cup\gamma_{s,i}\cup\beta_1\cup\beta_2$ be the closed curve such that 
\[
\beta_1(s)=e^{sX^{(n)}}\circ\phi(-\varepsilon,\cdots,-\varepsilon) \quad\text{ and }\quad \beta_1(s)=e^{sX^{(n)}}\circ\phi(\varepsilon,\cdots,\varepsilon)
\]
$\Gamma$ bounds a disk that is defined by 
\[
D(s,u)=e^{sX^{(n)}}\circ \gamma_{0,i}(u).
\]
We recall that from the definition of $\eta^{(n)}$, the set of points where $\eta^{(n)}$ is given by the iterates of the boundaries of the section $f^{T_n}\mathcal S$ therefore since $D$ is transversal to the flow then we can partition the disk $D$ into finite number of disks $D_1,\cdots, D_\ell$ such that $\eta^{(n)}$ is differentiable in $int(D_i)$ and $int(D_i)\cap int(D_j)=\emptyset$ for $i\neq j$. Using Stokes Theorem  we have
\[
\int_\Gamma\eta^{(n)}=\sum_{i=1}^\ell\int_{\Gamma_i}\eta^{(n)}=\sum_{i=1}^\ell\int_{D_i}d\eta_i^{(n)}
\]
Using \eqref{eq:tn} and the fact that each $D_i$ is tangent to $X^{(n)}$ we have
\begin{equation}\label{eq:Gamma1}
\left|\int_\Gamma\eta^{(n)}\right|\leq \sum_{i=1}^\ell|D_i|\|i_{X^{(n)}}\circ d\eta^{(n)}\|\leq C|D| \theta^{n/2}.
\end{equation}
We have the following
\begin{equation}\label{eq:Gamma2}
\int_{\gamma_{s,i}}\eta^*=\int_\Gamma\eta^{(n)}+\int_\Gamma(\eta^*-\eta^{(n)})+\int_{\beta_1}\eta^*-\int_{\beta_2}\eta^*.
\end{equation}
Since $X^{(n)}\to X_*\in\ker(\eta^*)$ then we have 
\begin{equation}\label{eq:Gamma3}
\lim_{n\to\infty}\int_{\beta_1}\eta^*=\lim_{n\to\infty}\int_{\beta_2}\eta^*=0
\end{equation}
From \eqref{eq:conv} we have $\eta^{(n)}\to\eta^*$  which implies
\begin{equation}\label{eq:Gamma4}
\lim_{n\to\infty}\int_\Gamma(\eta^*-\eta^{(n)})=0
\end{equation}
Substituting \eqref{eq:Gamma1}, \eqref{eq:Gamma3} and \eqref{eq:Gamma4} into \eqref{eq:Gamma2} we get the result.
\end{proof}
We observe that for every $q=\mathcal W_n(s,u_1,\cdots, u_m)$
 we have that 
\[
T_q\mathcal W_n=span\{X^{(n)}, De^{sX^{(n)}}\frac{d\gamma_{s,1}}{du},\cdots, De^{sX^{(n)}}\frac{d\gamma_{s,m}}{du}\}
\]
Therefore using the above claim we have and the fact that $De^{sX^{(n)}}$ is converging from Lemma \ref{lem:DeX} then we have 
\[
T_q\mathcal W_n\to\ker(\eta^*)\quad\text{ as }\quad n\to\infty.
\]
Thus $\mathcal W_n$ converges to a submanifold that is tangent to $\ker(\eta^*)$. This implies that $\ker(\eta^*)$ is integrable and since $F^*$ and $F$ are topologically equivalent, we get Theorem \ref{thm:main}.
\end{proof}

\section{Special synchronisation equation}\label{Sec:synceq}
The purpose of this section is to rewrite the quantities involved in \eqref{eq:synceq} in terms of the time change $\psi$ and the $1$-form in the first item of Definition \ref{def:sync}. 
Throughout this section, we suppose that we have a codimension-one Anosov flow $F=\{f^t, t\in\mathbb{R}\}$ generated by a $C^1$ vector field $Z$ whose stable bundle is generated by a $C^1$ vector field $X$.
We recall that we have a globally defined $1$-forms $\eta$ with the properties that
\begin{equation}\label{eq:eta}
(f^{t})^*\eta=\eta\quad\text{ and }\quad \eta(Z)\equiv1\quad \forall t\in\mathbb{R}.
\end{equation}
Since the stable bundle is one dimensional and generated by a vector field $X$ then the weak unstable bundle $\bE^{cu}$ is defined by a $1$-form $\omega$ with the properties that
\begin{equation}\label{eq:omega}
\ker \omega = \bE^{cu},
\quad\text{ and }\quad
\omega(X)\equiv1.
\end{equation}
We also recall that the weak unstable bundle being of codimension-one is $C^1$ and since $X$ is also a $C^1$ vector filed therefore the $1$-form $\omega$ is $C^1$. 
 Given a $C^{2}$ positive function $\psi: M\to(0,\infty)$, we define the function $h_\psi: M\to\mathbb{R}$ as
\[
h_\psi:=\frac{1}{\psi}\int_{0}^{\infty}d(\psi\circ f^t)(X)dt.
\]
By \eqref{eq:xi}, the stable bundle of the flow $F_\psi$ is given by 
\[
X_\psi=X+h_\psi Z.
\]
\begin{lem}\label{lem:LineL}
There exists a continuous line bundle $L\in\bE^u$ and a $C^1$ differential $1$-form $\alpha_0$ with $\alpha_0(Z)\equiv1$ such that \(
|\alpha_0|_{L}|>\frac{1}{2}.
\)
\end{lem}
\begin{proof}
Let $L\in\bE^{u}$ be a $C^{\theta}$ line bundle; this can be obtained by first considering a vector field  $\widetilde Z$ in $\bE^{cu}$ that approximates $Z$ but nowhere equal and then project this vector field onto $\bE^u$.
Let $p\in M$, in a neighborhood of $p$, we can find $\alpha_p$ solve the following equation
\[
\mathfrak L_Z\alpha_p\equiv0\quad \alpha_p(Z)(p)=1 \quad \alpha_p (v^u)(p)=1\quad \text{ and }\quad \alpha(\widetilde Z)(p)=0
\]
where $v^u\in L(p)$ defined by $Z(p)=\widetilde Z(p)+v^u$. Since $Z$ is $C^1$, $\alpha_p$ is also $C^1$.
This implies that there exists $\delta_p>0$ such that $\alpha_p(v^u)>\frac{1}{2}$ for $v^u\in L(q)$ with $Z(q)=\widetilde Z(q)+v^u$  and   $q\in B(p,\delta_p)$. This defines an open cover $\{B_1,\cdots, B_n\}$ such that in each $B_i$ there exists a $C^{1}$ differential $1$-form $\alpha_i$  such that $\|\alpha_i|_{L}\|>\frac{1}{2}$. We take a partition of unity $\{\rho_1,\cdots,\rho_n\}$ subject to the partition and define 
$$
\alpha=\sum_{i=1}^n\rho_i\alpha_i
$$
It is easy to see that 
\(
\|\alpha|_{L}\|>\frac{1}{2}
\)
which implies the result.

\end{proof}
Let $\alpha_0$ given by the above Lemma, and we define ${\alpha}$ by
\[
{\alpha}:=\alpha_0-\alpha_0(X)\omega.
\]
Since $X$ and $\omega$ are $C^{1}$ then $\alpha$ defines a $C^{1}$ form on $TM$. 
It is easy to see that 
\begin{equation}\label{eq:alpha1}
 \alpha(Z)\equiv1\quad\text{ and }\quad X\in\ker(\alpha).
\end{equation}
Given a $C^2$ positive function $\psi$, we let $\alpha^\psi$ be the $1$-form defined by 
\[
\alpha^\psi:=\psi(\alpha-h_{\psi}\omega).
\]
Then it easy to see that $X_\psi\in\ker(\alpha^\psi)$ and $\alpha^\psi(Z_\psi)\equiv1$. The main purpose of proving the Verjovsky conjecture is to solve the following equation 
\begin{equation}\label{eq:Verj}
\mathfrak{L}_{X_\psi}\alpha^{\psi}=i_{X_\psi}\circ d\alpha^\psi\equiv0.
\end{equation}

\begin{rem}\normalfont
We first observe that since $X_\psi\in\ker(\alpha^\psi)$ then we have $\mathfrak{L}_{X_\psi}\alpha^\psi(X_\psi)\equiv0.$   In addition, since the stable bundle of the Anosov flow generated by $Z_\psi=Z/\psi$  is given by $X_\psi$ then $[X_\psi, Z_{\psi}]=a_\psi X_\psi$ for some function $a_\psi$, which using Cartan Formula and the fact that $\alpha^\psi(Z_\psi)\equiv1$ imply $\mathfrak{L}_{X_\psi}\alpha^\psi(Z_\psi)=0$. Therefore to construct $\psi$ such that $\mathfrak{L}_{X_\psi}\alpha^\psi\equiv0$, we can focus on anihilating $\mathfrak{L}_{X_\psi}\alpha^\psi|_{\bE^u}$.
\end{rem}

By the definition of the Lie derivative, writing $\beta^\psi:=\alpha-h_\psi\omega$ gives
\[
\mathfrak{L}_{X_\psi}\alpha^{\psi}=X_\psi(\psi)\cdot\beta^\psi+\psi\mathfrak{L}_{X_\psi}\beta^\psi
\]

Using Cartan Formula and the fact that $X_\psi\in\ker(\beta^\psi)$ we have
\(
\mathfrak{L}_{X_\psi}\beta^\psi=i_{X_{\psi}}\circ d\beta^\psi.
\)
From the definition of $\beta^\psi,$ we have $d\beta^\psi=d\alpha-dh_\psi\wedge\omega-h_\psi d\omega$.
Therefore we have
\[
\mathfrak{L}_{X_\psi}\beta^\psi(v^u)= d\beta^\psi(X_\psi,v^u)=dh_\psi(v^u)+h_\psi\cdot(d\alpha(Z,v^u)-d\omega(X,v^u))+d\alpha(X,v^u)
\]
Thus, solving \eqref{eq:Verj}, is equivalent to finding a positive $C^2$ function $\psi$ such that 

\begin{equation}
\left(\frac{X_\psi(\psi)}{\psi}\alpha+dh_\psi+h_\psi\cdot(i_Z\circ d\alpha-i_X\circ d\omega)+i_X\circ d\alpha\right)|_{\bE^u}\equiv0
\end{equation}

For notational purposes we write
\[
\Uptheta(\psi):=X_\psi(\psi)\alpha+\psi\cdot dh_\psi+\psi\cdot h_\psi\cdot(i_Z\circ d\alpha-i_X\circ d\omega)+\psi i_X\circ d\alpha
\]
Observe that \eqref{eq:Verj} is equivalent to finding a $C^2$ positive function with 
\begin{equation}\label{eq:sync}
\Uptheta(\psi)|_{\bE^u}\equiv0.
\end{equation}
To simplify the notations, we write $g_\psi=\psi\cdot h_\psi=\int_0^\infty d(\psi\circ f^t)dt$ to have
\[
\Uptheta(\psi)=X(\psi)\alpha+ dg_\psi+g_\psi\cdot(-\frac{d\psi}{\psi}+\frac{Z(\psi)}{\psi}\alpha+i_Z\circ d\alpha-i_X\circ d\omega)+\psi i_X\circ d\alpha
\]
We will also need the following notation
\begin{equation}\label{eq:Upthetabar}
\overline\Uptheta(\psi):=dg_\psi+g_\psi\cdot(-\frac{d\psi}{\psi}+\frac{Z(\psi)}{\psi}+i_Z\circ d\alpha-i_X\circ d\omega)+\psi i_X\circ d\alpha.
\end{equation}

\begin{rem}\normalfont
We emphasize that, for the purpose of proving the Verjovsky Conjecture, we will not solve \eqref{eq:Verj} in this strong formulation. However we will construct a sequence of function that nearly solve the equation near pieces of sections as required in the definition of special synchronisability.
\end{rem}

\section{The time change}\label{sec:time}
This section is devoted to prove the following
\begin{prop}\label{prop:spesync}
Let $M$ be a smooth manifold of dimension at least four and $F=\{f^t: t\in\mathbb{R}\}: M\to M$ be a $C^1$ codimension-one volume preserving Anosov flow such that the stable bundle is generated by a $C^1$ vector field $X$ with the property $Df^tX=e^{-t}X$ for every $t\in\mathbb{R}$. Then $F$ is specially synchronisable. 
\end{prop}

The proof of Proposition \ref{prop:spesync} is done in Section \ref{sec:tos} under a more general and technical result in Proposition \ref{prop:loc}. For the rest of this section, we suppose that we have an Anosov flow that satisfies the condition of Proposition \ref{prop:spesync}, i.e., that has the Parry's synchronisation.

\subsection{General stragtegy}
We first remark that Equation \eqref{eq:Verj} is not a pointwise equation as the term $g_\psi$ contains information of along the forward orbit. This makes the equation difficult to be solved locally. \eqref{eq:Verj} can be written as
\begin{equation}\label{eq:gen1}
X(\psi)\alpha+ dg_\psi+g_\psi\cdot(-\frac{d\psi}{\psi}+\frac{Z(\psi)}{\psi}\alpha+i_Z\circ d\alpha-i_X\circ d\omega)+\psi i_X\circ d\alpha\equiv0.
\end{equation}
The special synchronisability condition suggests that it is enough to solve the above equation on the certain pieces of sections of the flow. As we mentioned previously the term 
\begin{equation}\label{eq:gen}
dg_\psi+g_\psi\cdot(-\frac{d\psi}{\psi}+\frac{Z(\psi)}{\psi}\alpha+i_Z\circ d\alpha-i_X\circ d\omega)
\end{equation}
makes it difficult to solve the problem locally. In other words, solving the equation 
\begin{equation}\label{eq:gen2}
X(\psi)\alpha+\psi i_X\circ d\alpha\equiv0
\end{equation}
on a piece of section of the flow is relatively much easier as the solution can be given by the general characteristic method of solving partial differential equation. The only difficulty would be that it is an overdetermined equation as we need it to be satisfied on every vector. The way we overcome this difficulty is to observe that since $X(\psi)\alpha+\psi i_X\circ d\alpha$ is a $1$-form, it has a kernel therefore basically there exists only one direction that we should care about. This is roughly saying that if we try to solve \eqref{eq:gen2} inductively by approximation; say we define a sequence $\psi_n$ that approximate the solution, at the step $n+1$, we can only consider the equation when applied to the vector that is off the kernel of  $X(\psi_{n})\alpha+\psi_n i_X\circ d\alpha$: this is the main motivation of Lemma \ref{lem:defV}.

In order to approximate a solution of \eqref{eq:gen1} with the above described method, we observe that since we only need the solution near the pieces of sections, given a function $\psi$, we can make a small perturbation $\widetilde\psi$ near the section such that $\widetilde\psi$ is closer to a solution of \eqref{eq:gen1} while the quantites in \eqref{eq:gen} are controlled as the perturbation is done in a very short tubular neighborhood of the section. This is done in Proposition \ref{prop:loc}. In Lemma \ref{lem:gvg}, we give an estimate of the terms in \eqref{eq:gen} that suggests that \eqref{eq:gen1} and \eqref{eq:gen2} are almost equivalent when we only need an approximation of a solution on pieces of sections.

One more, we emphasize that we are not going to solve \eqref{eq:gen1}, but the idea is to use an approximation that nearly solves it and is good enough to defined the $1$-form  that gives integrability of the stable and unstable bundle of the limit flow.
\subsection{Time change on sections}\label{sec:tos}
In this section, we suppose that we are given an admissible section $\mathcal S=\{\mathcal D_i\}_{i=1}^N$ and for notational purposes, for $\tau>0$ we write
\[
\mathcal S^{(\tau)}=\{\mathcal U^{(\tau)}_i, i=1,\cdots, N\}
\quad\text{ where }
\quad
\mathcal U_i^{(\tau)}:=\bigcup_{|t|<\tau}f^t\mathcal D_i.
\]

\begin{prop}\label{prop:loc}
Given a function $\psi\in\mathfrak D$, there exists a constant $\Cl{con:main1}>0$ that only depends on $M, \mathcal S, F, \psi, \alpha$ such that for every $\varepsilon>0$, there exists  $T_\varepsilon, \tau>0$  such  that for $T>T_\varepsilon$  there exists $\widetilde\psi_T\in\mathfrak D$ such that 
\begin{equation}
e^{-T}\leq \tau^2\quad\text{ and }\quad\|\psi-\widetilde\psi_T\|_{\mathfrak D}, \|\Uptheta(\widetilde\psi_T)|_{f^T\mathcal S^{(\tau/2)}}\|\leq \Cr{con:main1}\varepsilon.
\end{equation}
\end{prop}

\begin{proof}[Proof Proposition \ref{prop:spesync}]
The first requirement in the definition of special synchronisation is satisfied by the $1$-form $\alpha$ that is constructed in \eqref{eq:alpha1}. We  let $\theta\in(0,1)$. We define $\psi_0\equiv3/2$ and let $\psi_1$ be the function defined by Proposition \ref{prop:loc} for $\psi=\psi_0$ and $\varepsilon=\Cr{con:main1}^{-1}(\psi_0)\varepsilon_1$. Then $\eqref{eq:tau}$ is trivially satisfied. Equation \eqref{eq:synceq} follows from the observation in Section \ref{Sec:synceq} that $i_{X_{\psi}}\circ d\alpha^{\psi}=\Uptheta(\psi).$ Inductively given $\psi_{n-1}$, we construct $\psi_n$ from Proposition \ref{prop:loc} with $\varepsilon=\Cr{con:main1}^{-1}(\psi_{n-1})\theta^n$. As in the first step the sequence of functions $\{\psi_n, n\geq1\}$ satisfies the definition of special synchronisation. 
\end{proof}

The rest of this section is devoted to the proof of Proposition \ref{prop:loc}. In Section \ref{sec:defper}, we give the definition of $T_\varepsilon$ and $\widetilde \psi_T$. In Section \ref{sec:properties}, we prove Proposition \ref{prop:loc}.

\subsubsection{Definition of the perturbation}\label{sec:defper}
Since $\Uptheta(\psi)$ is a $1$-form then, to control its norm, it is enough to control is in certain direction that is off the kernel. The motivation of the following Lemma is to define the direction that we need to care about in order to estimate $\|\Uptheta(\psi)\|$.
 \begin{lem}\label{lem:defV}
There exists $\delta>0$ such that for every $\psi\in\mathfrak D$, there exists a continuous vector filed $V\in\bE^u$   with $\alpha(V)=1$ and $\|V\|\leq\delta^{-1}$ such that
\[
\measuredangle(V(q),\ker(\overline\Uptheta(\psi))_q)\geq\delta\quad\text{ whenever }\quad \bE^u_q\not\subset\ker(\overline\Uptheta(\psi))_q.
\] 
\end{lem}

\begin{proof}
First of all, we recall that from Lemma \ref{lem:LineL}, there exists a line bundle $L\subset\bE^u$ such that $|\alpha|_{L}|>1/2$. Then there exists $\delta>0$ such that $|\alpha|_{C_{\delta, L}}|>1/2$ where $C_{\delta, L}$ is a cone in the unstable bundle of angle $2\delta$ around $L$.
We let 
\[
N=\{q\in M: \bE^u_q\not\subset\ker(\overline\Uptheta(\psi))_q\}.
\]
It is easy to see that since we can choose a vector field $\widetilde V\in\bE^u$ such that $\measuredangle(\widetilde V,\ker(\overline\Uptheta(\psi)))>\delta$. We define define $V_1$ as a projection of $\widetilde V$ onto the cone $C_{\delta, L}$. This gives a required vector field in $N$. We then extend $V_1$ as a vector field in the closure of $N$. Recall that this vector field $V_1$ satisfies $\alpha(V_1)=1$ and $\|V_1\|\leq\delta^{-1}$. This vector field is therefore easily extend to continuous vector field in $M$ since in $M\setminus N$, we can take any extension.
\end{proof}

\begin{cor}\label{cor:basis}
For every $q\in M$, there exists a set of vectors $\{v_1,\cdots, v_m\}$ that spans $\bE^u_q$ such that 
\[
\alpha(v_i)=1,\quad \|v_i\|\leq\delta^{-1}, \quad \overline\Uptheta(\psi)(V-v_i)=0 \quad\text{ and }\quad \measuredangle(v_i,v_j)>\delta/2\quad\text{ for} i\neq j.
\] 
\end{cor}

\begin{proof}
If $\bE^u_q\subset\ker(\overline\Uptheta(\psi))_q$ then any set of vectors in $C_\delta$ whose pairewise angle is bigger than $\delta/2$ such that the results holds. If $\bE^u_q\not\subset\ker(\overline\Uptheta(\psi))_q$ then since $\measuredangle(V,\ker(\overline\Uptheta(\psi)))>\delta/2$, we can choose $m-1$ vectors in $ v_2,\cdots,  v_m\in C_{\delta, L}$ such that $ \overline\Uptheta(\psi)(V-v_i)=0$ with the properties that $\measuredangle(v_i,v_j),\measuredangle(V,v_i) >\delta$. We then choose $v_1=V(q)$ get the required family of vectors.
\end{proof}

\paragraph{Definition of $T_\varepsilon$:}
Since the vector field $V$ is continuous and $\psi\in\mathfrak D$ then $\Uptheta(\psi)(V)$ is continuous. By the classical method of mollifying non smooth function, we  can choose $\iota>0$  small enough such that if $\Phi$ is the $\iota$-mollification of the function $-\Uptheta(\psi)(V)$ then we have
\begin{equation}\label{eq:mol1}
\|\Phi+\Uptheta(\psi)(V)\|<\varepsilon.
\end{equation}
Moreover, there exists a constant $\Cl{con:Phi1}>0$ that only depend on $M$ such that 
\begin{equation}
\|\Phi\|_{C^r}\leq\Cr{con:Phi1}\iota^{-r}.
\end{equation}

We recall that from the definition of admissible section, the closure of the local sections are pairwise disjoint then there exists $\sigma>0$ such that the $\sigma$-neighborhoodof the pieces of sections are pairwise disjoint, i.e.,
\begin{equation}\label{eq:disj}
\{q\in M: d(q,\mathcal D_i)<\sigma\}\cap \{q\in M: d(q,\mathcal D_j)<\sigma\}=\emptyset\quad\text{ for }\quad i\neq j.
\end{equation}
For every $i=1,\cdots, \ell$, we define a $\sigma$-thickening  of the local section $\mathcal D_i$ as and local section $\widetilde{\mathcal D}_1$ with the following properties:
\begin{enumerate}
\item $\mathcal D_i\subset\widetilde{\mathcal D}_i$,
\item $\widetilde{\mathcal D}_i$ is admissible disk,
\item $d(\partial \widetilde{\mathcal D}_i, \partial {\mathcal D}_i)\geq\sigma$.
\end{enumerate}
Let  $\lambda_u>1$ be defined such that $\|Df^{-t}|_{\bE^u}\|\leq \lambda_u^{-t}$ for all $t>0$. We also recall that using the fact that $F$ preserves a volume form and $\dim(\bE^s)=1<\dim(\bE^u)$ then we have the following bunching condition 
\begin{equation}\label{eq:bunching1}
e^{-t}\|Df^t|_{\bE^u}\|\leq e^{-\zeta t}
\end{equation}
for some $\zeta>0$ and $t$ large enough.

Let $\tau>0$ and $T_\varepsilon$ such that 
\begin{equation}\label{eq:conss}
e^{-T_\varepsilon}<\lambda_u^{-T_\varepsilon}<e^{-\zeta T_\varepsilon}<\tau^2<\min\{\iota^4, \sigma^4\}
\end{equation}

For each $i$, we define a flow box around $\mathcal D_i$ of height $\tau$ by:
\[
\widetilde{\mathcal U}_i^{(\tau)}:=\bigcup_{|t|<\tau}f^t\widetilde{\mathcal D}_i.
\]
Since $\tau<\sigma$, by \eqref{eq:disj} we have
\[
\widetilde{\mathcal U}^{(\tau)}_i\cap\widetilde{\mathcal U}^{(\tau)}_j=\emptyset.
\]
 We consider the following ordinary differential equation
\begin{equation}\label{eq:rhoc1}
X(\overline\rho_{h,i})=e^{-T}\Phi\circ f^T\text{ on } \widetilde{\mathcal U}^{(\tau/2)}_i\quad\text{ and }\quad\overline\rho_{h,i}(\widetilde{\mathcal L}_i)\equiv0
\end{equation}
or, equivalently
\begin{equation}
X(\overline\rho_{h,i}\circ f^{-T})=\Phi_i\text{ on } f^T\widetilde{\mathcal U}^{(\tau/2)}_i\quad\text{ and }\quad\overline\rho_{h,i}(\widetilde{\mathcal L}_i)\equiv0.
\end{equation}
where 
\[
\widetilde{\mathcal L}_i:=\bigcup_{|t|<\tau/2}f^t{\mathcal L}_i.
\]

Equation \eqref{eq:rhoc1} is a first order linear PDE that has a unique solution given by the  
classical characteristic method of solving PDE:
\[
\overline\rho_{h,i}\circ e^{sX}(q)=e^{-T}\int_0^s\Phi\circ f^T\circ e^{\sigma X}(q)d\sigma\quad\text{ for }\quad q\in \widetilde{\mathcal L}_i\text{ and } s\in[-1,1].
\]
or, equivalently,
\begin{equation}\label{eq:rhoc0}
\overline\rho_{h,i}\circ f^{-T}\circ e^{se^{-T}X}(f^Tq)=e^{-T}\int_0^s\Phi\circ e^{\sigma e^{-T}X}(f^Tq)d\sigma\quad\text{ for }\quad q\in\widetilde{\mathcal L}_i\text{ and } s\in[-1,1].
\end{equation}
The submanifold $\mathcal L_i$ being transversal to the vector field $X$ gives a definition of a function $s_{}:\widetilde{\mathcal U}^{(\tau/2)}_i\to\mathbb{R}$ by 
\[
e^{-s_qX}(q)\in\widetilde{\mathcal L}_i.
\] 
With this function, we can rewrite the solution of Equaiton \eqref{eq:rhoc1} as 
\begin{equation}\label{eq:rhoc0}
\overline\rho_{h,i} (q)=e^{-T}\int_0^{s_q}\Phi\circ e^{(s-s_q) e^{-T}X}(f^Tq)ds.
\end{equation}
Since $\widetilde{\mathcal L}_i$ is $C^1$ then there exists a  constant $\Cl{con:s}>0$ that only depends on $X$ and $\mathcal L_i$ such that 
\begin{equation}
\|s_{(.)}\|_{C^1}\leq \Cr{con:s}.
\end{equation}
Let $\beta_i: \widetilde{\mathcal U}^{(\tau/2)}_i\to[0,1]$ be the standard bump function such that 
\[
\beta_i\equiv1\quad\text{ in }\quad \mathcal D_i\quad\text{ and }\quad \beta_i\equiv0\quad\text{ in }\quad \widetilde{\mathcal D}_i\setminus\mathcal D_i
\]
Moreover, there is a constant $\Cl{con:beta}>0$ that only depend on $M$ such that  
\begin{equation}\label{eq:beta}
\|\beta_i\|_{C^r}\leq\Cr{con:beta}\sigma^{-r}\quad\text{ for every }\quad r>0.
\end{equation}
We define $\rho_{h,i}:  {\mathcal U}^{(\tau/2)}_i\to \mathbb{R}$ by 
\[
\rho_{h,i}:=\beta_i\cdot\overline{\rho}_{h,i}.
\]
\begin{rem}
The motivation of using the bump function is to define a function that is supported in $\widetilde{\mathcal D}_i$ and 
\begin{equation}
X(\rho_{h,i})=e^{-T}\Phi\circ f^T\quad\text{ in }\quad \mathcal U^{(\tau/2)}_i
\end{equation}
\end{rem}
Let $\rho_v:\mathbb{R}\to[0,1]$ be the standard bump function that satisfes 
\begin{equation}\label{eq:rhov}
\text{supp}(\rho_v)\in(-\tau,\tau)\quad \rho_v(0)=1\quad\text{ in  }\quad (-\tau/2,\tau/2)\quad\text{ and }\quad \|\rho_v\|_{C^r}\leq \Cr{con:beta}\tau^{-r}.
\end{equation}
For every $i$, we define a function $\rho_{v,i}:\widetilde{\mathcal U}^{(\tau)}_i\to[0,1]$  by 
\[
\rho_{v,i}(q)=\rho_v(t_q)
\]
where $t_q$ is defined by $f^{-t_q}q\in\widetilde{\mathcal U}^{(\tau/2)}_i$. We observe that the function $t_{(.)}: \widetilde{\mathcal U}^{(\tau)}_i\to(-\tau,\tau)$ is $C^1$ and there exists a constant $\Cl{con:t}>0$ that only depend on $\widetilde{\mathcal D}_i$ and $Z$ such that 
\begin{equation}\label{eq:t}
\|t_{(.)}\|_{C^1}\leq\Cr{con:t}.
\end{equation}
We define the function $\rho:M\to\mathbb{R}$ by 
\[
\rho:=\sum_i\rho_i\quad\text{ where }\quad\rho_i:=\rho_{h,i}\rho_{v,i}.
\]
We define $\widetilde\psi_T:M\to\mathbb{R}$ by 
\begin{equation}
\widetilde\psi_{T}=\psi+\rho\circ f^{-T}.
\end{equation}
To simplify the notations, in all estimates in the following, we write $\widetilde \psi$ to mean $\widetilde\psi_T$.

\subsubsection{Estimates of the perturbation}\label{sec:properties}
This section is devoted to estimate the function $\rho$ defined in the previous section.
\begin{lem}
There exists a constant $\Cl{con:rhoh1}>0$ such that for every $i=1,\cdots, N$ we have the following
\begin{equation}\label{eq:XZ}
\|X(\rho_i\circ f^{-T})\|_{C^0}\leq\Cr{con:rhoh1}(\|\Uptheta(\psi)(V)\|_{C^0}+\varepsilon), \quad\|Z(\rho_i\circ f^{-T})\|_{C^0}\leq \Cr{con:rhoh1}\tau^{-1}e^{-T}(\|\Uptheta(\psi)(V)\|_{C^0}+\varepsilon)
\end{equation}
If $V\in\bE^u$ is a unit vector then we have
\begin{equation}\label{eq:V}
\|V(\rho_{i}\circ f^{-T})\|_{C^0}\leq\Cr{con:rhoh1}\tau^{-1}\lambda_u^{-T}
\end{equation}
\begin{equation}\label{eq:VX}
\|VX(\rho_{i}\circ f^{-T})\|_{C^0}\leq\Cr{con:rhoh1}\iota^{-1}.
\end{equation}
\end{lem}

\begin{proof}
For every $i$, by definition of $\rho_i$, we have the following
\[
\begin{aligned}
X(\rho_i\circ f^{-T})(f^Tq)&=e^{T}(X(\rho_{v,i})\cdot\rho_{h,i}+\rho_{v,i}\cdot X(\rho_{h,i}))(q)\\
&=e^T(X(t_q)\cdot\rho_{v,i}'(t_q)+\rho_{v,i}(t_q)\cdot X(\rho_{h,i})(\pi_i q))
\end{aligned}
\]
Since $X$ is tangent to $\widetilde{\mathcal D}_i$ we have $X(t_q)=0$ then, using that $\rho_{h,i}=\beta_i\cdot\overline\rho_{h,i}$ we have
\[
\|X(\rho_i\circ f^{-T})\|_{C^0}\leq e^{T}(\|\beta_i\|_{C^1}\cdot\|\overline\rho_{h,i}\|_{C^0}+\|\beta_i\|_{C^0}\cdot\|X(\overline\rho_{h,i})\|_{C^0})
\]
Using \eqref{eq:beta}, \eqref{eq:rhoc1} and \eqref{eq:rhoc0}, we have
\begin{equation}\label{eq:xproof}
\|X(\rho_i\circ f^{-T})\|_{C^0}\leq2\Cr{con:beta}\sigma^{-1}\|\Phi\|\leq 2\Cr{con:beta}\sigma^{-1}(\|\Uptheta(\psi)(V)\|_{C^0}+\varepsilon)
\end{equation}
where the last inequality uses \eqref{eq:mol1}. To estimate $\|Z(\rho_i\circ f^{-T})\|_{C^0}$, since $Z(\rho_{h,i})=0$, we can write
\begin{equation}\label{eq:zproof}
\|Z(\rho_i\circ f^{-T})\|\leq\|\rho_{v,i}\|_{C^1}\|\rho_{h,i}\|_{C^0}\leq \Cr{con:beta}\tau^{-1}e^{-T}(\|\Uptheta(\psi)(V)\|_{C^0}+\varepsilon)
\end{equation}
where the last inequality uses \eqref{eq:rhov} and \eqref{eq:rhoc0}. Equations \eqref{eq:xproof} and \eqref{eq:zproof} give the estimate of \eqref{eq:XZ} for any $\Cr{con:rhoh1}>\max\{\Cr{con:beta}, 2\Cr{con:beta}\sigma^{-1}\}$.
For $V\in\bE^u_{f^Tq}$ is a unit vector for some $q\in\mathcal U_i^{(\tau)}$, then  $V\in T_{f^Tq}f^T\mathcal U_i^{(\tau)}$. In particular there exists a unit vector  $V_0\in\bE^u_q$ such that $V=Df^{T}V_0/\|f^{T}V_0\|$ then we have 
 $V_T=Df^{-T}V=V_0/\|f^{T}V_0\|$ to have
\[
\begin{aligned}
V(\rho_i\circ f^{-T})(f^Tq)&=(V_T(\rho_{v,i})\cdot\rho_{h,i}+\rho_{v,i}\cdot V_T(\rho_{h,i}))(q)\\
&=V_T(t_q)\cdot\rho_{v,i}'(t_q)+\rho_{v,i}(t_q)\cdot D\pi_iV_T(\rho_{h,i})(\pi_i q)
\end{aligned}
\]
Since $V\in\bE^u$, we have $\|V_T\|\leq \lambda_u^{-T}$ and using \eqref{eq:t} we have $\|V_T(t_q)\|\leq \Cr{con:t} \lambda_u^{-T}$. To estimate $D\pi_iV_T(\rho_{h,i})$, we use its definition to have
\[
\|D\pi_iV_T(\rho_{h,i})\|\leq \lambda_u^{-T}(\|D\pi_iV_0(\beta_i)\|_{C^0}\|\overline\rho_{h,i}\|_{C^0}+\|\beta_i\|_{C^0}\|D\pi_iV_0(\overline\rho_{h,i})\|).
\]
From the definition of  $\overline\rho_{h,i}$ in \eqref{eq:rhoc0},   if $Y$ is a unit vector field tangent to $\mathcal D_i$ then we have
$\|Y(\overline\rho_{h,i})\|\leq e^{-T}\|s_{(.)}\|_{C^1}\|\Phi\|_{C^0}+e^{-T}\|Y_T\|\|\Phi\|_{C^1}\leq1$ where the last inequality uses the bunching condition in \eqref{eq:bunching1} and the relationship between the constants in \eqref{eq:conss} . Substituting into the above gives
\begin{equation}
\|V(\rho_i\circ f^{-T})\|_{C^0}\leq\lambda_u^{-T}
\end{equation}
Differentiating the first displayed line in this proof and using $X(t_q)=0$ gives
\[
\begin{aligned}
VX(\rho_i\circ f^{-T})(f^Tq)&=e^T(V_T(t_q)\rho_{v,i}'(t_q)\cdot X(\rho_{h,i})(\pi_i q)+\rho_{v,i}(t_q)\cdot D\pi_iV_TX(\rho_{h,i})(\pi_i q))
\end{aligned}
\]
Using \eqref{eq:rhoc1} we have $e^T\|X(\rho_{h,i})\|_{C^0}\leq\|\Phi\|_{C^0}$, by the previous argument we have  $|V_T(t_q)\rho_{v,i}'(t_q)|\leq \Cr{con:t}\lambda_u^{-T}\iota^{-1}$ and $\|D\pi_iV_TX(\rho_{h,i})\|_{C^0}\leq\|\Phi\|_{C^1}$. Substituting this into the previous displayed line and using the relationship between the constants in \eqref{eq:conss} we have the desired estimate.
\end{proof}

\begin{lem}\label{lem:gvg}
There exists a constant $\Cl{con:gvg}>0$ such that
 if $V\in\bE^u$ is a unit norm vector field we have
\begin{equation}\label{eq:Vg}
\|g_{\rho\circ f^{-T}}\|_{C^0}, \|Vg_{\rho\circ f^{-T}}\|_{C^0} \leq\Cr{con:gvg}\sqrt{\tau}.
\end{equation}
\end{lem}
\begin{proof}
We recall all the quantities we are going to estimate need information along the full forward orbit. To do the estimate, for $q\in M$, we consider the splitting of the forward orbit of $q$ by 
\[
0\leq t_1<t_1+2\tau<t_2<t_2+2\tau<\ldots<t_n<t_n+2\tau<\ldots
\]
such that 
\[
f^{t_i}\in\bigcup_j\widetilde{\mathcal U}^{(\tau)}_j \quad\text{ and }\quad f^{t}\notin \bigcup_j\widetilde{\mathcal U}^{(\tau)}_j \quad\text{ for }\quad t\in(t_i+\tau, t_{i+1})
\]
The times $t_i'$s correspond to the return times to some $\mathcal U^{(\tau)}$ and any piece of the the orbit of $q$ for $t\in(t_i+2\tau, t_{i+1})$. We observe that by \eqref{eq:conss}, we have 
\[
|t_{i+1}-t_i-2\tau|>\sigma/2.
\]
Since $\rho\circ f^{-T}$ is supported in $\bigcup_j\widetilde{\mathcal U}^{(\tau)}_j $ then we have
\[
g_{\overline\rho}(q)=\sum_{i}\int_{t_i}^{t_i+\tau}e^{-t}X(\overline\rho\circ f^{-T})\circ f^t(q)dt=\sum_{i}e^{-t_i}\int_{0}^{\tau}e^{-t}X(\overline\rho\circ f^{-T})\circ f^{t_i+t}(q)dt.
\]
And similarly if $V\in\bE^u_q$, we have

\[
V(g_{\overline\rho})(q)=\sum_{i}\int_{t_i}^{t_i+\tau}e^{-t}V_tX(\overline\rho\circ f^{-T})\circ f^t(q)dt=\sum_{i}e^{-t_i}\int_{0}^{\tau}e^{-t}V_{t_i+t}X(\overline\rho\circ f^{-T})\circ f^{t_i+t}(q)dt.
\]
where $V_t=Df^tV$.
By definition of $\rho$ we have
\[
X(\rho\circ f^{-T})=\sum_j\rho_{v,i}\circ f^{-T}X(\tilde\rho_{h,i}\circ f^{-T})
\]
Using \eqref{eq:XZ} and \eqref{eq:conss} we have
\[
\|g_{\overline\rho}(q)\|\leq 2\tau(\Cr{con:rhoh1}\sigma^{-1}\|\Uptheta(\psi)\|_{C^0}+\varepsilon)\sum_ie^{-t_i}\leq2\tau\frac{\Cr{con:rhoh1}\sigma^{-1}\|\Uptheta(\psi)\|_{C^0}+\varepsilon}{1-e^{-\sigma}}\leq\Cr{con:gvg}\sqrt{\tau}.
\]
That gives the estimate of $\|g_{\rho\circ f^{-T}}\|_{C^0}$. For the $C^1$ norm, we take any unit norm vector field $V\in T_qM$ and use the same splitting as above to have:

\[
V(g_{\overline\rho})(q)=\sum_{i}\int_{t_i}^{t_i+\tau}e^{-t}V_tX(\overline\rho\circ f^{-T})\circ f^t(q)dt=\sum_{i}e^{-t_i}\int_{0}^{\tau}e^{-t}V_{t_i+t}X(\overline\rho\circ f^{-T})\circ f^{t_i+t}(q)dt.
\]
where $V_t=Df^tV$. Then, similar to the previous estimate, using the bunching condition \eqref{eq:bunching1} we have
 \[
\|V(g_{\overline\rho})(q)\|\leq 2\tau\iota^{-2}\sum_ie^{-\zeta t_i}\leq\frac{2\tau\iota^{-2}}{1-e^{-\zeta\sigma}}\leq\Cr{con:gvg}\sqrt{\tau}.
\]
\end{proof}

\begin{cor}There exists a constant $\Cl{con:upth}>0$ such that 
if $V\in\bE^u$ is a unit vector field then we have
\begin{equation}\label{eq:Up1}
\left\|\overline\Uptheta(\widetilde\psi)(V)-\overline\Uptheta(\psi)(V)\right\|\leq\Cr{con:upth}\sqrt{\tau}
\end{equation}
\end{cor}
\begin{proof}
We recall that 
\begin{equation}\label{eq:Upthetabar}
\overline\Uptheta(\psi):=dg_\psi+g_\psi\cdot(-\frac{d\psi}{\psi}+\frac{Z(\psi)}{\psi}\alpha+i_Z\circ d\alpha-i_X\circ d\omega)+\psi i_X\circ d\alpha.
\end{equation}

Then we have
\begin{equation}\label{eq:bar1}
\begin{aligned}
\overline\Uptheta(\widetilde\psi)-\overline\Uptheta(\psi)&=(dg_{\widetilde\psi}-dg_\psi)+g_{\psi}(-\frac{d\widetilde\psi}{\widetilde\psi}+\frac{d\psi}{\psi}+\frac{Z(\widetilde\psi)}{\widetilde\psi}\alpha-\frac{Z(\psi)}{\psi}\alpha)\\
&+(g_{\widetilde\psi}-g_{\psi})(-\frac{d\psi}{\psi}+\frac{Z(\psi)}{\psi}+(i_Z\circ d\alpha-i_X\circ d\omega))
+(\widetilde\psi-\psi) i_X\circ d\alpha.\\
\end{aligned}
\end{equation}
We are going to estimate each term of the above equality.
We first observe that by \eqref{eq:Vg}, the first term is estimated by
\begin{equation}\label{eq:bar2}
\left\|(dg_{\widetilde\psi}-dg_\psi)(V) \right\|=\left\|dg_{\rho}(V) \right\|\leq\Cr{con:gvg}\sqrt{\tau}.
\end{equation}
Similarly, using that $\psi>1$ and \eqref{eq:Vg}  we have 
\begin{equation}\label{eq:bar3}
\left\|(g_{\widetilde\psi}-g_{\psi})(-\frac{d\psi}{\psi}+\frac{Z(\psi)}{\psi}+(i_Z\circ d\alpha-i_X\circ d\omega))\right\|\leq\Cr{con:gvg}\sqrt{\tau}(2\|\psi\|_{\mathfrak D}+\|\alpha\|_{C^1}+\|\omega\|_{C^1})
\end{equation}
Using \eqref{eq:rhoc0} we have
\begin{equation}\label{eq:bar4}
\|(\widetilde\psi-\psi) i_X\circ d\alpha\|\leq e^{-T}\|\Phi\|\|\alpha\|_{C^1}
\end{equation}
for the remaining term, we write
\[
\frac{d\widetilde\psi}{\widetilde\psi}-\frac{d\psi}{\psi}=\frac{\psi d\widetilde\psi-\widetilde\psi d\psi}{\psi\widetilde\psi}=\frac{\psi d(\widetilde\psi-\psi)+(\psi-\widetilde\psi)d\psi}{\psi\widetilde\psi}
\]
and 
\[
\frac{Z(\widetilde\psi)}{\widetilde\psi}-\frac{Z(\psi)}{\psi}=\frac{\psi Z(\widetilde\psi)-\widetilde\psi Z(\psi)}{\psi\widetilde\psi}=\frac{\psi Z(\widetilde\psi-\psi)+(\psi-\widetilde\psi)Z(\psi)}{\psi\widetilde\psi}
\]
Using \eqref{eq:XZ}, \eqref{eq:V} and the relationship between the constants in \eqref{eq:conss} we have
\begin{equation}\label{eq:bar5}
\left\|\frac{d\widetilde\psi}{\widetilde\psi}-\frac{d\psi}{\psi}\right\|\leq2\|V(\rho)\|_{C^1}+\|\rho\circ f^{-T}\|_{C^0}\|\psi\|_{\mathfrak D}\leq \tau
\end{equation}
and 
\begin{equation}\label{eq:bar6}
\left\|\frac{Z(\widetilde\psi)}{\widetilde\psi}-\frac{Z(\psi)}{\psi} \right\|\leq2\|Z(\rho)\|_{C^1}+\|\rho\circ f^{-T}\|_{C^0}\|\psi\|_{\mathfrak D}\leq \tau
\end{equation}
Substituting \eqref{eq:bar2}, \eqref{eq:bar3} \eqref{eq:bar4} \eqref{eq:bar5} and \eqref{eq:bar6} into \eqref{eq:bar1} gives the desired estimate.
\end{proof}

\begin{proof}[Proof of Proposition \ref{prop:loc}]
The estimate of the $\|\psi-\widetilde\psi\|_{\mathfrak D}$ follows directly from \eqref{eq:XZ}, \eqref{eq:V}, \eqref{eq:Vg} and the relationship between the constants \eqref{eq:conss}.

To estimate $\|\Uptheta(\widetilde\psi)\|$, we will estimate $|\Uptheta(\widetilde\psi)(v_i)|$ where $\{v_1,\cdots, v_m\}$ is the basis given by Corollary \ref{cor:basis}. 
 Using that $X(\rho)(f^Tq)=\Phi$ for $q\in \mathcal S^{(\tau/2)}$ and the fact that $\alpha(v_i)=1$ we have
\[
\begin{aligned}
\Uptheta(\widetilde\psi)(v_i)&= X(\psi)+\Phi+\overline\Uptheta(\widetilde\psi)(v_i)\\
&=\Uptheta(\psi)(V)+\Phi+\overline\Uptheta(\widetilde\psi)(v_i)-\overline\Uptheta(\psi)(V)\\
&=\Uptheta(\psi)(V)+\Phi+\overline\Uptheta(\widetilde\psi)(v_i)-\overline\Uptheta(\widetilde\psi)(v_i)+\overline\Uptheta(\psi)(v_i-V)
\end{aligned}
\]

Then using \eqref{eq:mol1}, \eqref{eq:Up1}  and Corollary \ref{cor:basis} we have
\begin{equation}
\left|\Uptheta(\widetilde\psi)(v_i)\right|\leq 2\Cr{con:upth}\varepsilon
\end{equation}
Using Corollary \ref{cor:basis} once again  we have
\[
\|\Uptheta(\widetilde\psi)\|\leq 2\Cr{con:upth}\delta^{-1}\varepsilon^2\max_i|\alpha(v_i)|\leq 2\Cr{con:upth}\delta^{-2}\varepsilon
\]
which gives the estimate of the proposition.
\end{proof}


\begin{thebibliography}{10}

\bibitem{Ano69}
\newblock D. V. Anosov.
\newblock Geodesic flows on closed Riemannian manifolds curvature. 
\newblock {\em Proceedings of the Bteklov Institute of Mathmetics.} vol. 90
(1967) (A. M. S. translation, 1969).


\bibitem{Asa08}
\newblock M. Asaoka.
\newblock On Invariant volume of codimension-one Anosov flows and the Verjovsky conjecture.
\newblock{\em Invent. Math. } 174:435--462, 2008. Erratum \newblock{\em Invent. Math. } 178:449, 2009.

\bibitem{BL13}
\newblock O. Butterley, C. Liverani.
\newblock Robustly invariant sets in fiber contracting bundle flows. 
\newblock {\em J. Mod. Dyn.}, 7(2):255--267, 2013. 


\bibitem{BW19}
\newblock O. Butterley, K. War.
\newblock Open sets of exponentially mixing Anosov flows. 
\newblock {\em J. Eur. Math. Soc.} 22, 2253-2285, 2020.


\bibitem{Fra70}
\newblock J. Franks.
\newblock  Anosov diffeomorphisms.
\newblock{\em Proc. of Symposia in Pure Math.}, 61--93, 1970.
 
 \bibitem{FW80}
  \newblock J. Franks and R. Williams.
 \newblock Anomalous Anosov flows.
 \newblock{\em Global Theory and Dynamical Systems, Springer Lecture Notes 819.} Springer, Berlin, 1980.

 \bibitem{FMT07}
 M.~Field, I.~Melbourne, and A.~T\"or\"ok.
 \newblock Stability of mixing and rapid mixing for hyperbolic flows.
 \newblock {\em Ann. of Math. (2)}, 166(1):269--291, 2007.
  
\bibitem{Ghy89}
\newblock E. Ghys.
\newblock Codimension-one Anosov flows and suspensions.
\newblock{\em Lecture Notes in Mathematics} vol. 1331, 59--72, Springer-Verlag, 1989.

\bibitem{Hart82}
\newblock P. Hartman.
\newblock Ordinary Differential Equations.
\newblock 2nd Edition, Birkhauser, 1982.

\bibitem{HPS77}
\newblock M. Hirsch, C. Pugh, and M. Shub.
\newblock  Invariant manifolds.
\newblock{\em Lecture Notes in Mathematics.} Vol.~583, Springer-Verlag, Berlin-New York, 1977.

\bibitem{LTW16}
\newblock S. Luzzatto, S. T\"ureli, K. War.
\newblock  Integrability of continuous bundles. 
\newblock{\em Journal für die reine und angewandte Mathematik (Crelle's Journal)} 752 (2019)

\bibitem{Par86}
\newblock W. Parry.
\newblock  Synchronisation of canonical measures for hyperbolic attractors.
\newblock{\em Comm. Math. Phys.} 106(2):267--275, 1981.

\bibitem{Pla81}
\newblock J. Plante.
\newblock  Anosov flows, transversally affine foliations and a conjecture of Verjovsky.
\newblock{\em J. London Math. Soc.} 23(2):359--362, 1981.

\bibitem{Pla72}
\newblock J. Plante.
\newblock  Anosov flows.
\newblock{\em Amer. J. of Math. } 94:729--754, 1972.

\bibitem{New70}
\newblock S. Newhouse.
\newblock  On codimension-one Anosov diffeomorphisms.
\newblock{\em Amer. J. of Math. } 92:761--770, 1970.

\bibitem{Sim96}
\newblock S. Simi\'c.
\newblock  Lipschitz distributions and Anosov flows.
\newblock {\em Proc. Amer. Math. Soc.}, 124(6):1869--1877, 1996.

\bibitem{Sim97}
\newblock S. Simi\'c.
\newblock Codimension-one Anosov flows and a conjecture of Verjovsky.
\newblock {\em Ergod. Th. \& Dynam. Sys.}, 17:1211--1231, 1997.

\bibitem{Sim05}
\newblock S. Simi\'c.
\newblock Volume preserving codimension-one Anosov flows in dimensions greater than three are suspensions.
\newblock {\em Unpublished, later withdrawn by the author.} arXiv:math/0508024, 2005.

\bibitem{Ver74}
\newblock A.~Verjovsky.
\newblock codimension-one Anosov flows.
\newblock {\em Bol. Soc. Mat. Mexicana (2)}, 19(2):49-77, 1974.
\end{thebibliography}
\end{document}